%% file: main.tex
\documentclass[11pt]{article}
\usepackage{graphicx} 

\usepackage{fancyhdr}
\usepackage{isomath}
\usepackage{mathtools} 
\usepackage{amsbsy,amsmath,amsthm}
\usepackage{amssymb}
\usepackage{amscd,amsfonts}
\usepackage{bigints}
\usepackage{graphicx}
\usepackage{verbatim}
\usepackage{euscript}
\usepackage{alltt}
\usepackage{stmaryrd}
\usepackage{relsize}
\usepackage{enumerate}
\usepackage{url}
\usepackage[stable]{footmisc}
\usepackage{breakurl}
\usepackage{hyperref}
\usepackage{comment}

\usepackage[maxbibnames=9]{biblatex}
\usepackage[font=small,labelfont=bf]{caption}
\usepackage[font=small,labelfont=bf]{subcaption}
\usepackage{float}
\usepackage{mwe}
\usepackage{algorithm}
\usepackage{algpseudocode}
\usepackage{xcolor}
\usepackage[super]{nth}
\usepackage{soul}

\DeclareGraphicsExtensions{.eps,.pdf}
\input{temp-definitions}

\usepackage[margin=1in]{geometry}

\date{}
\begin{document}
\title{A convex variational principle for the necessary conditions of classical optimal control}

\author{Amit Acharya\thanks{Department of Civil \& Environmental Engineering, and Center for Nonlinear Analysis, Carnegie Mellon University, Pittsburgh, PA 15213, email: acharyaamit@cmu.edu.} $\qquad \qquad \qquad$ Janusz Ginster\thanks{Weierstrass Institute, Mohrenstrasse 39, 10117 Berlin, Germany, email: ginster@wias-berlin.de.} }

\maketitle
\begin{abstract}
\noindent A scheme for generating a family of convex variational principles is developed, the Euler-Lagrange equations of each member of the family formally corresponding to the necessary conditions of optimal control of a given system of ordinary differential equations (ODE) in a well-defined sense. The scheme is applied to the Quadratic-Quadratic Regulator problem for which an explicit form of the functional is derived, and existence of minimizers of the variational principle is rigorously shown. It is shown that the Linear-Quadratic Regulator problem with time-dependent forcing can be solved within the formalism without requiring any nonlinear considerations, in contrast to the use of a Riccati system in the classical methodology.  

Our work demonstrates a pathway for solving nonlinear control problems via convex optimization.

\end{abstract}

\section{Introduction}
Optimal control theory for ODEs is a vast and well-developed subject at a level of maturity where many textbooks have been written on it - an excellent introduction, e.g., is \cite{evans1983introduction}, and it is beyond the scope (and neither the intent) of this work to provide an extensive review of the literature on the subject. Our modest goal here is to describe a particular viewpoint for attacking the equations describing some necessary conditions satisfied by optimal solutions of the problem, an approach which, to our knowledge, is new. The proposed scheme results in an unconstrained \emph{convex} variational principle for a mapping of time alone, taking values in $\R^{n+m+n}$ (where $t \mapsto x(t) \in \R^n$ describes the state, $t \mapsto u(t) \in \R^m$ is the control, and $t \in [0,T] \subset \R$ is time). This is in contrast to the Hamilton-Jacobi-Bellman (HJB) equation based approach to the optimal control problem which involves solving a first-order, nonlinear scalar PDE for the value function on a subset of $\R^{n+1}$; for $n$ large, this is a manifestation of the so-called `curse of dimensionality.'  Thus, our approach is expected to have some practical relevance, albeit that it solves necessary conditions of the problem, a feature also shared by the Pontryagin Maximum Principle (PMP). The PMP is extensively used in important practical applications, and our work contributes to efforts to bring general nonlinear control problems within the purview of convex optimization techniques \cite{malyuta2022convex} and, in general, to the body of work treating control problems as one of optimization \cite{cuthrell1987optimization}.

The Pontryagin Maximum Principle has recently been incorporated as a soft constraint into a Machine Learning scheme called PMP-net for optimal control problems  \cite{kamtue2024pontryagin}, much in the spirit of physics-informed-neural networks \cite{raissi2019physics}; a Least Squares objective defined from the PMP equations is added as an additional component to the Loss function used to train the scheme. Given a set of equations to be solved, it is understood that the solutions to the Euler-Lagrange equations of the least squares functional generated from `squaring' the equations can generate spurious solutions to the given set of equations, while our proposed duality scheme does not share this shortcoming \cite{action_2,sukumar2024variational}. Of course, it is understood that the least-squares strategy is a shortcoming only when the given equations constitute a `ground-truth' model - like the PMP for control problems - to which the ML trained trajectories should ideally comply. Thus, the developed dual variational principles in this paper have the potential to contribute to modern applications of control theory within the PMP-net paradigm by providing improved theory-informed training constraints. The proposed approach in this article arose in questions related to solving/approximating the (non)linear governing equations of problems in continuum mechanics by a duality based approach \cite{action_2,dual_cont_mech_plas,sga,KA1,KA2}, and is strongly related to the ideas of Hidden Convexity in PDE advanced by Y.~Brenier \cite{brenier2018initial, brenier_book}. 

An outline of the paper is as follows: in Sec.~\ref{sec:gen_form} we develop the differential algebraic system with boundary conditions that constitute necessary conditions of the optimal control problem of interest. In Sec.~\ref{sec:dual_vp} the derivation of the dual variational principle for the governing system developed in Sec.~\ref{sec:gen_form} is presented. The formalism is applied to the Quadratic-Quadratic Regulator problem in Sec.~\ref{sec:QQR} where the explicit form of the dual functional is obtained, along with its analog amenable to convex optimization. Sec.~\ref{sec:analysis} contains a proof of coercivity and boundedness from below of the dual functional for the QQR problem which, along with its guaranteed weak lower-semicontinuity by design, ensures the existence of a minimizer. Finally, Sec.~\ref{sec:LQR} contains an application to the classical Linear-Quadratic Regulator problem, with time-dependent forcing. A comparison of the proposed scheme with the classical solution protocol for the problem based on solving a nonlinear Riccati system is discussed, and an explicit solution to the dual problem for a simple example LQR is also derived. While we do not deal with free-final-time problems in this paper as well as state-space constraints, the general dual methodology developed here suggests that both of these problem features can be naturally accommodated.

A few words on notation: an overhead dot will represent a time derivative. We will always use the summation convention on repeated indices, unless otherwise specified. The notation $f|_a$ denotes the evaluation of the function $f$ at the argument $a$, i.e., $f|_a = f(a)$, and we will use both notations as convenient.

\section{The necessary conditions for optimality in the classical optimal control problem}\label{sec:gen_form}
The classical optimal control problem may be stated as (see, e.g., \cite{evans1983introduction}, \cite[Sec.~9.5,9.6]{luenberger_optimization}),
\begin{subequations}\label{eq:op_cont}
\allowdisplaybreaks
    \begin{align}
        & \mbox{for} \quad r: \R^n \times \R^m \times \R \to \R; \quad x: \R \to \R^n; \quad u: \R \to \R^m; \quad g:\R^n \times \R \to \R; \quad r,g \ \mbox{given} \notag\\
        & \mbox{maximize} \qquad  P[u]  = - \int_0^T r(x(t), u(t),t) \, dt - g(x(T),T) \tag{\ref{eq:op_cont}}\\
        & \mbox{subject to} \begin{cases}
            \dot{x}(t)  = f(x(t),u(t),t), \qquad f:\R^n \times \R^m \times \R \to \R^n  \ \mbox{given},\\
            x(0) = x^0 \in \R^n, \quad x^0 \ \mbox{specified}.
        \end{cases} \notag
    \end{align}
\end{subequations}
Here, $x$ is the state function and $u$ is the control.

Using the method of Lagrange multipliers with $p:\R \to \R^n$ as the co-state multiplier functions, one can seek to solve the unconstrained maximization problem given by
\begin{equation*}
    \mbox{maximize} \quad \tilde{P}[x,u,p] = - \int_0^T r\left(x|_t,u|_t,t \right) \ + \ p|_t \ \cdot \ \big(\dot{x}|_t - f(x|_t,u|_t,t) \big) \, dt \ - \ g(x|_T, T) \ - \ p|_0 \cdot (x|_0 - x^0).
\end{equation*}
The first variation of $\tilde{P}$ about a state $(x,u,p)$ in the direction $(\delta x, \delta u, \delta p)$ is given by
\begin{subequations}
\allowdisplaybreaks
\begin{align}
    \delta \tilde{P}|_{(\delta x, \delta u, \delta p)}[x,u,p] & = \int_0^T - \p_x r \cdot \delta x - \p_u r \cdot \delta u - \delta p \cdot (\dot{x} - f) - p \cdot \delta \dot{x} + p \cdot (\p_u f \cdot \delta u + \p_x f \cdot \delta x) \quad dt \notag\\
    & \qquad - (\p_x g)|_T \cdot \delta x|_t \ - \ \delta p|_0 \cdot x|_0  \ +  \ \delta p|_0 \cdot x^0 \ - \ p|_0 \cdot \delta x|_0 \notag\\
    & = \int_0^T \delta x \cdot \big(  - \p_x r + \dot{p} + p_j \p_x f_j \big) \ + \ \delta u \cdot \big(  - \p_u r + p_i \p_u f_i \big) \ + \ \delta p \cdot \big( - \dot{x} + f \big) \ dt   \notag\\
    & \qquad + \delta x|_T \cdot \big( - (\p_x g)|_{(x|_T,T)} - p|_T \big) \  - \ \delta x|_0 \cdot (-p|_0 + p|_0) \ + \ \delta p|_0 \cdot \big( x|_0 - x^0 \big). \notag
\end{align}
\end{subequations}
Consequently, a maximizer $(x,u,p)$ of $\tilde{P}$ would formally satisfy the following Euler-Lagrange equations:
\begin{subequations}\label{eq:control_EL}
\allowdisplaybreaks
    \begin{align}
        \dot{x}|_t - f\big(x|_t,u|_t, t\big) & = 0 \label{eq:x_evol}\\
        \dot{p}|_t \ -  \ \p_x r\big(x|_t,u|_t,t \big) \ + \ p_j|_t \, \p_x \, f_j \big(x|_t,u|_t,t\big) & = 0 \label{eq:p_evol}\\ 
        p_i \p_u f_i \big(x|_t,u|_t,p|_t\big) - \p_u r \big(x|_t,u|_t,p|_t \big) & = 0 \label{eq:H_crit}\\
        p|_T + \p_x g(x|_T,T) & = 0. \label{eq:p_bc}\\
        x|_0 -  x^0 & = 0 \label{eq:x_bc}
    \end{align}   
\end{subequations}
We note that defining the `control theory Hamiltonian' \cite[Sec.~4.3]{evans1983introduction} as
\[
H(x,u,p,t) : = p  \cdot f(x,u,t) - r(x,u,t),
\]
(\ref{eq:x_evol},\ref{eq:p_evol},\ref{eq:H_crit}) may be expressed as
\begin{subequations}
\allowdisplaybreaks
\begin{align}
     \dot{x} & = \p_p H \label{eq:H_1}\\
    \dot{p}  &= - \p_x H \label{eq:H_2}\\
     0  &= \p_u H \label{eq:H_3}
\end{align}
\end{subequations}
which are necessary conditions satisfied by a triple of functions $(x,u,p)$ satisfying the Pontryagin Maximum Principle (PMP) of Optimal Control (see, e.g., \cite{evans1983introduction}) for \eqref{eq:op_cont}; (\ref{eq:H_1},\ref{eq:H_2}) are exact statements of the PMP, while \eqref{eq:H_3} is a necessary condition satisfied by solutions of the PMP when there is enough smoothness for $\p_u H (x,u,p,t)$ to make sense.

\section{A dual variational principle for \eqref{eq:control_EL} and a convex optimization problem}\label{sec:dual_vp}
The equations \eqref{eq:control_EL} contain first order ODEs and are not easily converted to a second order system typical of Euler-Lagrange equations of variational principles with a local Lagrangian that is a function of first order derivatives in time and lower-order terms; we note that the Hamiltonian $H$ is linear in $p$ and hence not strictly convex in it, so that a Legendre transform cannot be implemented to obtain a corresponding Lagrangian. Moreover, the forward-in-time nature of the ($x$) state  evolution (\ref{eq:x_evol}, \ref{eq:x_bc}) and backward-in-time nature of the ($p$) co-state evolution (\ref{eq:p_evol}, \ref{eq:p_bc}) pose a significant challenge for practical approximation schemes for solving nonlinear control problems not suffering from the `curse of dimensionality,' as arises in a Hamilton-Jacobi-Bellman (HJB) formulation \cite[Secs.~5.1.2,5.1.3]{evans1983introduction} of the problem \cite{borggaard2020quadratic,krener,navasca2000solution,al1961optimal}. The goal here is to develop a variational approach to \eqref{eq:control_EL} which has the potential of addressing these issues. As well, it allows the inclusion (but not as a necessity) of `guiding' state, control, and co-state functions $(t \mapsto \bar{x}(t), t \mapsto \bar{u}(t), t \mapsto \bar{p}(t))$ in the formulation as parameters, such that obtained solutions may be expected to be close to these user-specified guiding functions.

We denote the \emph{primal} functions $U := (x,u,p)$ and consider arbitrarily specified (`designable') functions of time called \emph{base states} 
\[
\bar{U} := (\bar{x},\bar{u},\bar{p}): \R \supset [0,T] \to \R^n \times \R^m \times \R^n; \qquad t \mapsto \bar{U}(t).
\]
Also considered is a triple of \emph{dual} functions
\[
D := (\gamma, \mu, \lambda); \qquad \gamma, \lambda: \R \supset [0,T] \to \R^n, \qquad \mu: \R \supset [0,T] \to \R^m; \qquad \dee := (D, \dot{D})
\]
and a freely chosen (designable) \emph{auxiliary potential}
\[
\sch: \R^n \times \R^m \times \R^n \times \R^n \times \R^m \times \R^n \to \R
\]
that is assumed to have a positive-definite Hessian everywhere w.r.t its first $n+m+n$ arguments.

In terms of these ingredients, we define the \emph{pre-dual} functional
\begin{equation}\label{eq:pre_dual}
\begin{aligned}
    \widehat{S}_\sch[x,u,p, \gamma, \mu, \lambda] & = \int_0^T \Big( - x_i \dot{\gamma}_i - \gamma_i f_i(x,u,t) \\
    & \qquad \qquad - p_i \dot{\lambda}_i - \lambda_i \p_{x_i} r(x,u,t) + \lambda_i p_j \p_{x_j} f_i (x,u,t) \\
    & \qquad \qquad + \mu_\kappa p_i \p_{u_\kappa} f_i(x,u,t) - \mu_\kappa \p_{u_\kappa} r (x,u,t)  \\
    & \qquad \qquad  - \sch(x,u,p,\bar{x}, \bar{u}, \bar{p}) \ \Big) \, dt \\
    & \qquad \qquad - \gamma_i|_0 \, x^0_i \quad + \quad \gamma_i|_T \, x_i|_T \quad - \quad \lambda_i|_T \, \p_{x_i} g(x|_T,T)  \\
    & =: \int_0^T \scl_{\sch}(U,\dee, \bar{U}, t) \,dt \quad - \quad \gamma_i|_0 \, x^0_i \quad + \quad \gamma_i|_T \, x_i|_T  \quad - \quad  \lambda_i|_T \, \p_{x_i} g\big|_{(x|_T,T)} ,
\end{aligned}
\end{equation}
and require the choice of $\sch$ to be such that it enables the definition of a \emph{dual-to-primal} (DtP) mapping
\[
(\dee, \bar{U}, t) \mapsto U^\sch(\dee, \bar{U}, t) = \Big(x^\sch, u^\sch, p^\sch \Big)\Big|_{(\dee, \bar{U}, t)}; \qquad U^\sch: \R^{2n+(n+m+n)+1} \supset \mathcal{O} \to \R^{n+m+n} 
\]
such that
\begin{equation}\label{eq:dL_dU}
    \partial_U \scl_\sch \big( U^\sch(\dee, \bar{U}, t), \dee, \bar{U}, t) = 0
\end{equation}
is satisfied on the set $\mathcal{O}$. Thus, the choice of $\sch$ should allow `solving for $U$ in terms of $(\dee, \bar{U}, t)$' from \eqref{eq:dL_dU}.

One now defines the \emph{dual} functional, $S_\sch[\gamma, \mu,\lambda]$, as the one obtained by the substitution of the DtP mapping into the pre-dual functional:
\begin{equation}\label{eq:pont_dual}
    \begin{aligned}
        S_\sch[\gamma, \mu,\lambda] & := \widehat{S}_\sch \Big[x^\sch,u^\sch,p^\sch, \gamma, \mu, \lambda \Big] \\
        & = \int_0^T \scl_{\sch}\Big(U^\sch\big( \dee|_t, \bar{U}|_t, t \big),\dee|_t, \bar{U}|_t, t \Big) \,dt \\ 
        & \quad - \ \gamma_i|_0 \, x^0_i \quad + \quad  \gamma_i|_T \, x_i^\sch\Big(\dee|_T,\bar{U}|_T,T \Big)  \quad - \quad   \lambda_i|_T \, \p_{x_i} g \Big(x^\sch \big(\dee|_T,\bar{U}|_T,T \big),T \Big).
    \end{aligned}
\end{equation}
In the following, we use the notation
\[
 \quad t \mapsto \mathsf{d}(t) = \mathsf{d}|_t := (\dee|_t,\bar{U}|_t,t) ; \qquad  \mathsf{s}^\sch \big|_t := \Big( x^\sch \Big|_{ \mathsf{d}|_t}, u^\sch \Big|_{ \mathsf{d}|_t}, t \Big); \qquad \delta \dee := (\delta D, \dot{\overline{\delta D}}).
\]
The first variation of $S_\sch$ at a dual state $D$ in a direction $\delta D: [0,T] \to \R^{n+m+n}$ (denoted below by  $\delta S_\sch|_{\delta D} [D]$), constrained by the boundary conditions
\begin{subequations}\label{eq:dual_bc}
\begin{align}
   \gamma_i|_T \ - \ \lambda_j|_T \, \p_{x_i x_j} g \, \Big( x^\sch \big|_{ \mathsf{d}|_t} , T \Big)  & = 0 \label{eq:gamma_bc}\\
   \lambda|_0 = \ \mbox{specified arbitrarily, say } \lambda^0; \qquad \delta \lambda|_0 & = 0 \label{eq:lambda_bc}
\end{align}
\end{subequations}
and for which \eqref{eq:dL_dU} is satisfied in $[0,T]$, is given, after integration by parts and using \eqref{eq:dL_dU} and \eqref{eq:lambda_bc}, by
\begingroup
\allowdisplaybreaks
    \begin{align}
         \delta S_\sch\big|_{\delta D}[D] & = \int_0^T \bigg( \ \delta \gamma_i|_t \, \bigg\{ \dot{\overline{x^\sch_i \circ \mathsf{d}}}\Big|_t \ - \ f_i \Big( \mathsf{s}^\sch \big|_t \Big) \bigg\} \notag \\
        & \qquad \qquad + \delta \lambda_i|_t \, \bigg\{ \dot{\overline{p^\sch_i \circ \mathsf{d}}}\Big|_t  \ - \ \p_{x_i} r \Big( \mathsf{s}^\sch \big|_t \Big) \ + \ p^\sch_j\Big|_{ \mathsf{d}|_t} \, \p_{x_j} f_i \Big( \mathsf{s}^\sch \big|_t \Big) \bigg\} \notag\\
        & \qquad \qquad + \delta \mu_\kappa|_t \, \bigg\{ p^\sch_i \Big|_{ \mathsf{d}|_t}  \p_{u_\kappa} f_i \Big( \mathsf{s}^\sch \big|_t \Big) -  \p_{u_\kappa} r \Big(  \mathsf{s}^\sch \big|_t\Big)\bigg\}      \ \bigg) \, dt \label{eq:dual_first_var}\\
        & \quad + \ \delta \gamma_i|_0 \, \Big\{  x^\sch_i \Big|_{\mathsf{d}_0} - x_i^0 \Big\} \quad - \quad \delta \lambda_i|_T \, \Big\{ p^\sch_i\Big|_{ \mathsf{d}|_T}  + \p_{x_i} g \Big( x^\sch \Big|_{ \mathsf{d}|_T}, T \Big) \Big\} \notag\\
        & \quad + \ \delta \gamma_i|_T \, \Big\{  x^\sch_i \Big|_{ \mathsf{d}|_T} - x^\sch_i \Big|_{ \mathsf{d}|_T} \Big\} + \Big( \p_\dee \, x^\sch_i\Big|_{ \mathsf{d}|_T} \cdot \delta \dee \Big) \Big\{ \gamma_i|_T \ - \ \lambda_j|_T \, \p_{x_i x_j} g \, \Big( x^\sch \Big|_{ \mathsf{d}|_T} , T \Big) \Big\}. \notag
    \end{align}
\endgroup
The last line of \eqref{eq:dual_first_var} vanishes due to \eqref{eq:gamma_bc}; one term of the first curly bracket on the same line arises from the variation of the boundary term $\gamma_i|_T \, x_i|_T $; the other from the variation of the term $- x_i \dot{\gamma}_i$ in the Lagrangian $\scl_\sch$ after integration by parts. The five other expressions within curly brackets in \eqref{eq:dual_first_var} are the set of equations \eqref{eq:control_EL} with the substitution $U \to U^\sch$. Thus, by standard arguments, the Euler-Lagrange equations of the dual functional \eqref{eq:pont_dual} are the necessary conditions \eqref{eq:control_EL} for the classical optimal control problem \eqref{eq:op_cont}, utilizing an adapted change of variables defined by the DtP mapping $U^\sch$.

It is worth noting that if the the DtP mapping has the property that $U^\sch (\dee,\bar{U},t) = \bar{U}$ for $\dee = 0$ - this can be arranged in many circumstances by an appropriate choice of $\sch$, and examples are provided in Secs.~\ref{sec:QQR}-\ref{sec:LQR} - then, if $\bar{U}$ was a solution to \eqref{eq:control_EL} then $D= 0$ is a critical point of $S_\sch$.

Depending on the nonlinearity of the terminal cost function $g$, the final-time boundary condition \eqref{eq:gamma_bc} may pose a nonlinear constraint on the dual space of functions involved; for $g$ a quadratic form in its first argument, the boundary condition becomes a linear constraint.

\emph{In all that follows, we will assume that $g = x_i G_{ij} x_j$, with $G$ symmetric, positive semi-definite}.

With this choice, along with imposing the constraint 
\begin{equation}\label{eq:linear_dual_bc}
\gamma|_T - G \lambda|_T = 0
\end{equation}
(the b.~c.~\eqref{eq:gamma_bc}) for all functions $(\gamma, \lambda)$ in the set on which $\widehat{S}_\sch$ is defined, the functional takes the form
\[
\widehat{S}_\sch[x,u,p, \gamma, \mu, \lambda] = \int_0^T \scl_\sch(U,\dee, \bar{U}, t) \,dt \quad - \quad \gamma_i|_0 \, x^0_i.
\]

We now consider a slightly different (but related) dual functional $\tilde{S}_\sch$ given by
\begin{equation}
    \tilde{S}_\sch[D] = \sup_U \int_0^T \scl_\sch (U,\dee, \bar{U}, t) \,dt \quad - \quad \gamma_i|_0 \, x^0_i.
\end{equation}
It is reasonable that the maximization over $U$ can be moved into the integral and noting that \emph{ $\scl_\sch$ is necessarily affine in $\dee$},
\[
\tilde{\scl}_\sch (\dee, \bar{U}, t) := \sup_U \scl_\sch (U,\dee, \bar{U}, t)
\]
is \emph{convex in $\dee$}.

Moreover, note that if formulated in the correct function spaces then $\tilde{S}_H$ is the supremum over continuous, affine functions. In particular, $\tilde{S}_H$ is the supremum over weakly continuous functions and therefore lower-semicontinuous with respect to weak convergence. Thus, up to coercivity of $\tilde{S}_\sch[D]$ one may even expect to prove existence of minimizers for it in appropriate Sobolev spaces for the dual fields $(\gamma, \mu, \lambda)$ (as is done in Sec.~\ref{sec:analysis}). 

We note that for the class of dual fields for which the maximization (over $U$) to define $\tilde{S}_\sch$ has a unique maximizer, the DtP mapping $U^\sch$ is a well-defined object a.e. in $[0,T]$, and for such dual fields the value of $S_\sch$ and $\tilde{S}_\sch$ coincide. Since a minimizer of the convex $\tilde{S}_\sch$ functional is also its critical point (at least formally), such minimizers can be expected to be critical points of $S_\sch$, if they also belong to the set for which the DtP mapping is well-defined a.e~in $[0,T]$, i.e., under suitable hypotheses
\[
t \mapsto U^\sch\big(\dee[{\arg\!\min} \, \tilde{S}_\sch[D]](t)\big)
\]
can define a solution to \eqref{eq:control_EL}, which are necessary conditions satisfied by solutions to the problem of classical optimal control \eqref{eq:op_cont}. In the above, $\dee[D^*](t)$ denotes the evaluation, at time $t$, of $\dee$ constructed from the dual field $D^*$.

Thus, upon discretization of $\tilde{S}_\sch[D]$ by any Rayleigh-Ritz basis, we have a \emph{convex optimization} formulation of the necessary conditions \eqref{eq:control_EL} of the, in general nonlinear, classical optimal control problem.

\section{A dual convex optimization formulation for the Quadratic-\\Quadratic Regulator (QQR)}\label{sec:QQR}
The goal of this section is to develop the explicit formula for the dual functional for the Quadratic-Quadratic Regulator problem \cite{borggaard2020quadratic,krener} involving a quadratically nonlinear state evolution. Many important scientific problems fall within this class - the Lorenz system \cite{lorenz1963deterministic} and the Fermi-Pasta-Ulam-Tsingou problem \cite{fermi1955studies} in the ODE context, and (discretized versions of) the Euler and Navier-Stokes equations.

We consider the special case of quadratic running and terminal cost, as well as a quadratically nonlinear state evolution with linear dependence on the control $u$. The problem \eqref{eq:op_cont} is defined by:
\begin{subequations}
\allowdisplaybreaks
\label{eq:QQR_def}
\begin{align}
& \mbox{Given matrices} \quad B, G \in \R^{n \times n},  C \in \R^{m \times m}, M \in \R^{n \times n}, N \in \R^{n \times m}, F \in \R^{n \times n \times n} \notag\\
& B,C,G \ \mbox{symmetric, positive-semidefinite} \notag\\
& F \ \mbox{symmetric in last two indices}, \mbox{ i.e., } F a = 0, \forall \, a \in \R^{n \times n} \ \mbox{and skew-symmetric;} \notag\\
    & r(x,u,t) := \half (x_i B_{ij} x_j + u_\alpha C_{\alpha \beta} u_\beta) \notag\\
    & g(x,T) := x_iG_{ij} x_j \notag\\
    & f_i(x,u,t) := A_i(t) + M_{ij} x_j + N_{i\alpha} u_\alpha + \half x_r F_{irs} x_s, \mbox{ with given function } t \mapsto A(t)  \in \R^n \mbox{ and} \notag\\
    & \mbox{Given functions} \ t\mapsto \bar{x}(t), t \mapsto \bar{u}(t), t\mapsto \bar{p}(t) \mbox{ (with 0 a possible choice for any of } \bar{x}, \bar{u}, \bar{p}), \notag
\end{align}
\end{subequations}
the primal system is given by
\begin{equation}
    \label{eq:primal_QQR}
    \begin{aligned}
        \dot{x}_i & = A_i + M_{ij} x_j + N_{i\alpha} u_\alpha + \half x_r F_{irs} x_s \\
        \dot{p}_i & = B_{ij} x_j - p_j  M_{ji} - p_j F_{jri} x_r\\
        0 & = p_i N_{i \alpha} - u_\beta C_{\beta \alpha} 
    \end{aligned}
\end{equation}
on $[0,T]$ with boundary conditions
\begin{equation}
    \label{eq:primal_QQR_bc}
    \begin{aligned}
         p_i|_T + G_{ij} x_j|_T & = 0\\
    x_i|_0 & = x^0_i.
    \end{aligned}
\end{equation}
As the auxiliary potential $\mathcal{H}$ we choose the quadratic potential
\[
Q(x,u,p) := \half \Big( a_x |x - \bar{x}|^2 + a_u |u - \bar{u}|^2 + a_p |p - \bar{p}|^2 \Big), \qquad a_x, a_u, a_p > 0.
\]
Then the pre-dual functional is given by, c.f.~\eqref{eq:pre_dual},
\begin{equation*}
    \begin{aligned}
        \widehat{S}_Q[U,D] & =  \int_0^T \Big\{ - \gamma_i A_i - x_i \dot{\gamma}_i - \gamma_i M_{ij} x_j - \gamma_i N_{i \alpha} u_\alpha - \gamma_i \half F_{irs} x_r x_s \\
        & \qquad \qquad - p_i \dot{\lambda}_i - \lambda_i B_{ij} x_j + p_j M_{ji} \lambda_i + p_j F_{jri} x_r \lambda_i\\
        & \qquad \qquad +  p_i N_{i \alpha} \mu_\alpha -  u_\beta C_{\beta \alpha} \mu_\alpha\\
        & \qquad \qquad - \half \Big( a_x |x - \bar{x}|^2 + a_u |u - \bar{u}|^2 + a_p |p - \bar{p}|^2 \Big) \ \Big\} \, dt\\
        & \qquad - \gamma_i|_0 x^0_i \quad + \quad \gamma_i|_T x_i|_T \quad - \quad \lambda_i|_T G_{ij} x_j|_T,
    \end{aligned}
\end{equation*}
and after imposition of the constraint boundary condition \eqref{eq:gamma_bc}, the last two terms drop out. 

\emph{We denote the integrand as $\scl_Q$}. In order to generate the explicit formula for the dual QQR functional, it is efficient to write the Lagrangian $\scl_Q$ in terms of the factor $U - \bar{U}$ instead of $U$ and subsequently focus on the linear, quadratic, and `constant' terms in this factor. Then the linear and quadratic terms combine in a compact manner to deliver the final expression \eqref{eq:QQR_dual}. These steps are demonstrated below.

The DtP mapping $U^Q := \big(x^Q,u^Q,p^Q \big)$ is generated from the following set of conditions:
\begin{subequations}
    \label{eq:QQR_DtP}
    \allowdisplaybreaks
    \begin{align}
        \p_{x_i} \scl_Q &= 0: \quad - \Big( a_x \delta_{ir} + \gamma_k F_{kir} \Big) \,\Big(x^Q_r - \bar{x}_r \Big) + \lambda_k F_{jik} \Big( p^Q_j - \bar{p}_j \Big) \notag\\
        & \qquad \qquad \qquad \qquad \qquad \qquad = \dot{\gamma}_i + \gamma_k M_{ki} + \lambda_k B_{ki} + \gamma_k F_{kis} \bar{x}_s - \lambda_k F_{jik} \bar{p}_j \tag{\ref{eq:QQR_DtP}}\\
        \p_{p_i} \scl_Q & = 0: \quad - a_p \Big(p^Q_i - \bar{p}_i \Big) + \lambda_k F_{irk} \,\Big(x^Q_r - \bar{x}_r \Big) = \dot{\lambda}_i - M_{ik} \lambda_k - N_{i\alpha} \mu_\alpha - \lambda_k F_{irk} \bar{x}_r \notag\\
        \p_{u_\alpha} \scl_Q & = 0: \quad - a_u \Big(u^Q_\alpha - \bar{u}_\alpha \Big)  = \gamma_i N_{i \alpha} + C_{\alpha \beta} \mu_\beta. \notag
    \end{align}
\end{subequations}
We now introduce the notation
\begin{subequations}
    \allowdisplaybreaks
    \begin{align}
        & \calX_i|_\dee := \dot{\gamma}_i + \gamma_k M_{ki} + \lambda_k B_{ki} \ ; \qquad \bar{\calX}_i|_{(\dee,\bar{U})} := \gamma_k F_{kis} \bar{x}_s - \lambda_k F_{jik}\bar{p}_j \notag \\
        & \calP_i |_\dee := \dot{\lambda}_i - M_{ik} \lambda_k - N_{i \alpha} \mu_\alpha \ ; \qquad \bar{\calP}_i |_{(\dee,\bar{U})} := - \lambda_k F_{irk} \bar{x}_r \notag\\
        & \calU_\alpha |_\dee := \gamma_i N_{i \alpha} + C_{\alpha \beta} \mu_\beta  \notag \\
        \notag \\
        & \mathcal{J}|_{(\dee, \bar{U})} : = \begin{Bmatrix} \calX|_\dee + \bar{\calX}|_{(\dee, \bar{U})} \\ \\ \calP|_\dee + \bar{\calP}|_{(\dee, \bar{U})} \end{Bmatrix} \in \R^{2n}\; \qquad \K|_\dee := \begin{bmatrix} a_x I + (\gamma F) & -(F\lambda)^T   \\
                                                                  &   \\
                                                                  -(F \lambda) & a_p I \end{bmatrix} \in \R^{2n \times 2n} \ ; \notag\\
                                                                  &  I \mbox{ is the } n \times n \mbox{ identity matrix} \ ; \qquad F: \R^n \to \R^{n \times n} \ ; \qquad t \mapsto (\gamma(t) F) \in \R^{n \times n}, \ (\gamma F)_{ir} =\gamma_k F_{kir} \ \notag\\
        & t \mapsto U_*(t) := \begin{Bmatrix} x(t)  \\
                                    \\
                                  p(t)
                   \end{Bmatrix} \in \R^{2n} \ ; \qquad t \mapsto \bar{U}_*(t) = \begin{Bmatrix} \bar{x}(t) \\
                                                                \\
                                                             \bar{p}(t)
                                                        \end{Bmatrix} \in \R^{2n} \notag
    \end{align}
\end{subequations}
in terms of which the DtP mapping can be expressed as
\begin{subequations}\label{eq:QQR_DtP}
    \allowdisplaybreaks
    \begin{align}
        - \K_\dee \big(U^Q_* - \bar{U}^Q_*\big) & = \mathcal{J}|_{(\dee,\bar{U})} \notag\\
        - a_u (u - \bar{u}) & = \calU |_\dee. \tag{\ref{eq:QQR_DtP}}
    \end{align}
\end{subequations}
The Lagrangian $\scl_Q$ can be expressed as
\begin{subequations}\label{eq:mixed_QQR_L}
\allowdisplaybreaks
\begin{align}
    \scl_Q \big( U,\dee,\bar{U},t\big) & = - \gamma_i A_i -(x_i - \bar{x}_i) \dot{\gamma}_i - \bar{x}_i \dot{\gamma}_i - \gamma_i M_{ij} (x_j - \bar{x}_j) - \gamma_i M_{ij} \bar{x}_j - \gamma_i N_{i\alpha} (u_\alpha - \bar{u}_\alpha) - \gamma_i N_{i \alpha} \bar{u}_\alpha \notag\\
    & \quad - (p_i - \bar{p}_i) \dot{\lambda}_i - \bar{p}_i \dot{\lambda}_i - \lambda_i B_{ij} (x_j - \bar{x}_j) - \lambda_i B_{ij} \bar{x}_j + (p_j - \bar{p}_j) M_{ji} \lambda_i + \bar{p}_j M_{ji} \lambda_i \notag\\
    & \quad + (p_i - \bar{p}_i) N_{i \alpha} \mu_\alpha + \bar{p}_i N_{i \alpha} \mu_\alpha - \mu_\alpha C_{\alpha \beta} (u_\beta - \bar{u}_\beta) - \mu_\alpha 
    C_{\alpha \beta} \bar{u}_\beta \notag\\
    & \quad - \half \Big( a_x (x - \bar{x}_i) (x - \bar{x}_i) + a_u(u_\alpha - \bar{u}_\alpha) (u_\alpha - \bar{u}_\alpha) + a_p(p_i - \bar{p}_i)(p_i - \bar{p}_i) \Big) \notag\\
    & \quad - \half (x_r - \bar{x}_r) \gamma_i F_{irs} (x_s - \bar{x}_s) \notag\\
    & \quad - \half (x_r - \bar{x}_i) \gamma_iF_{irs} \bar{x}_s  - \half \bar{x}_r \gamma_i F_{irs} (x_s - \bar{x}_s) + \half \bar{x}_r \gamma_i F_{irs} \bar{x}_s \notag\\
    & \quad - \half \bar{x}_r \gamma_i F_{irs} \bar{x}_s - \half \bar{x}_r \gamma_i F_{irs} \bar{x}_s \notag\\
    & \quad + \lambda_i ( p_j - \bar{p}_j ) F_{jri} ( x_r - \bar{x}_r ) \notag \\
    & \quad + \lambda_i (p_j - \bar{p}_j ) F_{jri} \bar{x}_r + \lambda_i \bar{p}_j F_{jri} ( x_r - \bar{x}_r ) - \lambda_i \bar{p}_j F_{jri} \bar{x}_r \notag \\
    & \quad + \lambda_i \bar{p}_j F_{jri} \bar{x}_r + \lambda_i \bar{p}_j F_{jri} \bar{x}_r \notag \\
    & \qquad \qquad = \notag \\
    & \quad (x_i - \bar{x}_i) \big( - \dot{\gamma}_i - \gamma_k M_{ki} - \lambda_k B_{ki} - \gamma_k F_{kis} \bar{x}_s + \lambda_k \bar{p}_j F_{jik} \big) \notag \\
      & \quad - \half (x_i - \bar{x}_i) \big( a_x \delta_{ir} + \gamma_k F_{kir} \big) (x_r - \bar{x}_r) + \half (x_i - \bar{x}_i) F_{jik} \lambda_k (p_j - \bar{p}_j) \notag \\
      & \quad  + (p_i - \bar{p}_i) \Big( - \dot{\lambda}_i + M_{ik} \lambda_k + N_{i \alpha} \mu_\alpha + \lambda_k F_{irk} \bar{x}_r \Big) \notag \\
      & \quad  - \half a_p (p_i - \bar{p}_i) (p_i - \bar{p}_i) + \half (p_ i - \bar{p}_i) F_{irk} \lambda_k (x_r - \bar{x}_r) \notag\\
      & \quad  + (u_\alpha - \bar{u}_\alpha) (- \gamma_i N_{i \alpha} - C_{\alpha \beta} \mu_\beta ) - \half a_u (u_\alpha - \bar{u}_\alpha) (u_\alpha - \bar{u}_\alpha) \notag\\
      & \quad  + \bar{x}_i \big( - \dot{\gamma}_i - \gamma_k M_{ki} - \lambda_k B_{ki} - \gamma_k F_{kis} \bar{x}_s + \lambda_k \bar{p}_j F_{jik} \big) \notag \\
      & \quad  + \bar{p}_i \Big( - \dot{\lambda}_i + M_{ik} \lambda_k + N_{i \alpha} \mu_\alpha + \lambda_k F_{irk} \bar{x}_r \Big) \notag \\
      & \quad + \bar{u}_\alpha \Big( - \gamma_i N_{i \alpha} - C_{\alpha \beta} \mu_\beta \Big) \notag\\
      & \quad  + \Big( \half + \half \Big) \bar{p}_j F_{jri}\lambda_i  \bar{x}_r - \half \bar{x}_r \gamma_i F_{irs} \bar{x}_s - A_i \gamma_i  \notag \\ 
      & \qquad \qquad = \notag \\
      & \quad - \half \big( U_* - \bar{U}_* \big) \cdot \K|_\dee \big( U_* - \bar{U}_* \big) - \half a_u (u - \bar{u})\cdot (u - \bar{u}) \notag\\
      & \quad - \big( U_* - \bar{U}_* \big) \cdot \mathcal{J}|_{(\dee,\bar{U})} - (u - \bar{u}) \cdot \mathcal{U}_\dee \tag{\ref{eq:mixed_QQR_L}}\\
      & \quad - \bar{x} \cdot \left( \calX |_\dee + \half \bar{\calX}|_{(\dee, \bar{U})} \right) - \bar{p} \cdot \left( \calP |_\dee + \half \bar{\calP}|_{(\dee, \bar{U})} \right) - \bar{u} \cdot \calU |_\dee - A \cdot \gamma.\notag
\end{align}
\end{subequations}
 so that the QQR pre-dual functional may be written as
 \begin{subequations}\label{eq:mixed_QQR}
 \allowdisplaybreaks
 \begin{align}
      \widehat{S}_Q[U,D] & = - \half \int_0^T \Big\{ \big( U_* - \bar{U}_* \big) \cdot \K|_\dee \big( U_* - \bar{U}_* \big) - a_u (u - \bar{u})\cdot (u - \bar{u}) \Big\} \, dt\notag\\
      & \quad - \int_0^T \Big\{ \big( U_* - \bar{U}_* \big) \cdot \mathcal{J}|_{(\dee,\bar{U})} + (u - \bar{u}) \cdot \mathcal{U}_\dee \Big\} \, dt \tag{\ref{eq:mixed_QQR}}\\
      & \quad - \int_0^T \bigg\{ \bar{x} \cdot \left( \calX |_\dee + \half \bar{\calX}|_{(\dee, \bar{U})} \right) + \bar{p} \cdot \left( \calP |_\dee + \half \bar{\calP}|_{(\dee, \bar{U})} \right) + \bar{u} \cdot \calU |_\dee + A \cdot \gamma \bigg\} \, dt \notag\\
      & \quad - \gamma|_0 \cdot x^0. \notag
 \end{align}
 \end{subequations}

 Substituting for $U^Q$ from \eqref{eq:QQR_DtP} for $U$ in $\scl_Q$, we have
\begin{equation*}
    \allowdisplaybreaks
    \begin{aligned}
        \scl_Q \Big(U^Q|_{(\dee,\bar{U})},\dee, \bar{U}, t \Big) & = \bigg(- \half + 1 \bigg) \bigg( \mathcal{J}|_{(\dee,\bar{U})} \cdot \K \big|_\dee^{-1} \mathcal{J}|_{(\dee,\bar{U})} + \frac{1}{a_u} \calU \big|_\dee \cdot \calU \big|_\dee \bigg) \notag\\
        & \quad - \bar{x} \cdot \left( \calX |_\dee + \half \bar{\calX}|_{(\dee, \bar{U})} \right) - \bar{p} \cdot \left( \calP |_\dee + \half \bar{\calP}|_{(\dee, \bar{U})} \right) - \bar{u} \cdot \calU |_\dee - A \cdot \gamma
    \end{aligned}
\end{equation*}
so that the \emph{QQR dual functional} is given by
\begin{subequations}
    \label{eq:QQR_dual}
    \begin{align}
        S_Q[D] & = \int_0^T \scl_Q \Big(U^Q|_{(\dee,\bar{U})},\dee, \bar{U}, t \Big) \, dt - \ \gamma|_0 \cdot x^0 \notag\\
        & = \half \int_0^T \bigg( \mathcal{J}|_{(\dee,\bar{U})} \cdot \K \big|_\dee^{-1} \mathcal{J}|_{(\dee,\bar{U})} + \frac{1}{a_u} \calU \big|_\dee \cdot \calU \big|_\dee \bigg) \, dt \tag{\ref{eq:QQR_dual}}\\
        & \quad - \int_0^T \bigg\{ \bar{x} \cdot \left( \calX |_\dee + \half \bar{\calX}|_{(\dee, \bar{U})} \right) + \bar{p} \cdot \left( \calP |_\dee + \half \bar{\calP}|_{(\dee, \bar{U})} \right) + \bar{u} \cdot \calU |_\dee + A \cdot \gamma \bigg\} \, dt \ - \  \gamma|_0 \cdot x^0. \notag
    \end{align}
\end{subequations}

We note from the DtP mapping \eqref{eq:QQR_DtP} that, for the field $D = 0$ so that $\dee(t) = 0 \ \forall \ t \in [0,T]$, $U^Q = \bar{U}$, and recall that the E-L equation of $S_Q$ is the QQR primal equation set \eqref{eq:primal_QQR} with the replacement $U \to U^Q$. Thus, if $\bar{U}$ was a solution of \eqref{eq:primal_QQR}, then $D=0$ is a critical point of $S_Q$.

If we now define
\begin{equation} \label{eq: def tildeS}
\tilde{S}_Q[D] \ := \ \sup_{U} \, \widehat{S}_Q [U,D] \ = \ \int_0^T \sup_{U} \, \scl_Q (U, \dee, \bar{U}, t) \, dt \ - \ \gamma|_0 \cdot x^0
\end{equation}
\eqref{eq:mixed_QQR_L} implies that
\begin{equation*}
\sup_{U} \, \scl_Q (U, \dee, t, \bar{U}) = \begin{cases} \scl_Q \Big(U^Q|_{(\dee,\bar{U})},\dee, \bar{U}, t \Big) &\mbox{ for any } \dee  \text{ s.t. }  \K(\dee)  \text{ is positive semi-definite} \\ &\text{ and } \mathcal{J}|_{(\dee,\bar{U})} \in im(\K(\dee))\\
+ \infty  &\text{ otherwise}.
\end{cases}
\end{equation*}
Indeed, first note that $\K(\dee)$ is a symmetric matrix. In particular, the above maximization problem above is concave if $\K(\dee)$ is positive semi-definite. Hence, every critical point is a maximizer, i.e.~every solution of \eqref{eq:QQR_DtP} is a maximizer of $U \to \mathcal{L}_Q(U,\dee,\bar{U},t)$. Next, if $\K(\dee)$ is only positive semi-definite then the expression $\K(\dee)^{-1}$ has to be interpreted as the inverse of the linear mapping restricted to the orthogonal complement of its kernel. This is well-defined as all solutions of \eqref{eq:QQR_DtP} return the same maximal value for $\scl_Q$. If $0\neq \mathcal{J}|_{(\dee,\bar{U})} \notin im(\K(\dee))$ then its projection onto the kernel of $\K(\dee)$, call it $U$, is non-zero as $\K(\dee)$ is symmetric. Then note that $\mathcal{L}_Q(sU + \bar{U},\dee,\bar{U},t) \rightarrow +\infty$ as $s \to -\infty$. Similarly, one can argue if $\K(\dee)$ is not positive semi-definite, i.e.~if $\K(\dee)$ has a negative eigenvalue. Thus,
\begin{equation}\label{eq: tildeSQ}
\tilde{S}_Q [D] = \begin{cases} S_Q [D] \ &\mbox{ for } D \mbox{ s.t. } \K(\dee(t)) \mbox{ is positive semi-definite and $\mathcal{J}|_{(\dee(t),\bar{U}(t))} \in im(\K(\dee(t)))$} \\ &\text{ a.e.~in } [0,T] \\
+ \infty &\mbox{ otherwise}.
\end{cases}
\end{equation}
As already noted, $\tilde{S}_Q$ is a \emph{convex} functional of the dual fields $D$.

\section{Coercivity of the dual minimization problem for the Quadratic-Quadratic Regulator (QQR)}\label{sec:analysis}

In this section we will discuss the coercivity of the dual functional $\tilde{S}_Q$ as defined in \eqref{eq: def tildeS} on an appropriate subspace of $H^1((0,T); \R^n) \times H^1((0,T); \R^n) \times L^2((0,T); \R^m)$. 
Given $\lambda_0 \in \R^n$ we set
\[
\mathcal{A}:=\{(\gamma,\lambda,\mu) \in H^1((0,T); \R^n) \times H^1((0,T); \R^n) \times L^2((0,T); \R^m): \lambda(0) = \lambda^0, G \lambda(T) = \gamma(T)  \}.
\]
Moreover, we define the function $g: \R^n \times \R^n \times \R^n \times \R^n \times \R^m \to [0,\infty]$ as
\begin{align*}
g(\gamma,\alpha,\lambda,\beta,\mu) = \sup_{x,p \in\R^n, u \in \R^m} &\bigg[-x_j \cdot ( \alpha_j + \gamma_i M_{ij} + \lambda_i B_{ij} ) - p_j ( \beta_j -  M_{ji} \lambda_i + N_{ji} \mu_i ) \\ & - u_j ( N_{ij} \gamma_i + C_{ji} \mu_i)  - \frac12\left( a_x |x |^2 + a_p |p|^2 + a_u |u|^2 \right) \\ & - \frac12 \gamma_i F_{irs} x_r x_s + p_j F_{jri} x_r \lambda_i \bigg].
\end{align*}
As indicated in \eqref{eq: def tildeS} in this setting it can be shown rigorously that it holds for all $(\gamma,\lambda,\mu) \in \mathcal{A}$
\begin{align*}
\tilde{S}_Q[\gamma,\lambda,\mu] &= \sup_{x,p \in L^2((0,T);\R^n), u\in L^2((0,T);\R^m)} \hat{S}_Q[(x,u,p),(\gamma,\lambda,\mu)] \\ &=\int_0^T g(\gamma,\dot{\gamma},\lambda,\dot{\lambda},\mu) + A \cdot \gamma\, ds - \gamma(0)\cdot x^0. 
\end{align*}
In particular, since $\hat{S}_Q$ is affine in $(\gamma,\lambda,\mu)$ the functional $\tilde{S}_Q$ is lower-semicontinuous with respect to weak convergence in $H^1((0,T);\R^n) \times H^1((0,T);\R^n) \times L^2((0,T);\R^m)$.
Hence, the coercivity result of this section, Proposition \ref{prop: coer}, guarantees that the direct method of the Calculus of Variations can be applied to establish the existence of a minimizer of $\tilde{S}_Q$. By the argumentation in Section \ref{sec:dual_vp} such a minimizer is then (at least formally) a solution to the primal equations \eqref{eq:control_EL}, c.f.~also the discussion in \cite[Section 5]{sga}.

We start by proving the following lower bound for $g$.

\begin{proposition}\label{prop: lb g}
    It holds for all $(\gamma,\alpha,\lambda,\beta,\mu) \in \R^n \times \R^n \times \R^n \times \R^n \times \R^m$ that
    \begin{align*}
    g(\gamma,\alpha,\lambda,\beta,\mu) \geq &\frac{\left| \alpha + M^T \gamma + B^T \lambda \right|^2}{2 (a_x + |F| |\gamma| + |F| |\lambda|)}  + \frac{\left| \beta - M \lambda + N \mu \right|^2}{2 ( a_p + |F| |\lambda|)}  + \frac{\left| N^T \gamma + C \mu \right|^2}{2a_u} .
    \end{align*}
    For $F = 0$ it holds 
    \[
        g(\gamma,\alpha,\lambda,\beta,\mu) = \frac{\left| \alpha + M^T \gamma + B^T \lambda \right|^2}{2 a_x}  + \frac{\left| \beta - M \lambda + N \mu \right|^2}{2  a_p }  + \frac{\left| N^T \gamma + C \mu \right|^2}{2a_u} .
    \]
\end{proposition}
\begin{proof}
    We estimate using Young's inequality
    \begin{align*}
    g(\gamma,\alpha,\lambda,\beta,\mu) &\geq \sup_{x,p \in\R^n, u \in \R^m} \bigg[-x_j \cdot ( \alpha_j + \gamma_i M_{ij} + \lambda_i B_{ij} ) - p_j ( \beta_j -  M_{ji} \lambda_i + N_{ji} \mu_i ) \\ & \qquad - u_j ( N_{ij} \gamma_i - C_{ji} \mu_i)   - \frac12\left( a_x |x |^2 + a_p |p|^2 + a_u |u|^2 \right) - \frac12 (|\gamma| + |\lambda|) \, |F| \, |x|^2 \\ & \qquad -  \frac12 |\lambda| \, |F| \, |p|^2 \bigg] \\
    &= \frac{\left| \alpha + M^T \gamma + B^T \lambda \right|^2}{2 (a_x + |F| |\gamma| + |F| |\lambda|)}  + \frac{\left| \beta - M \lambda + N \mu \right|^2}{2 ( a_p + |F| |\lambda|)}  + \frac{\left| N^T \gamma - C \mu \right|^2}{2a_u} .
\end{align*}
For $F =0$ the inequality above is an equality.
\end{proof}

Before we state a coercivity result for $\tilde{S}_Q$, let us briefly introduce some notation.
Given a matrix $R \in \R^{d\times d}$ and $t \in \R$ we denote by $e^{tR} := \sum_{k=0}^{\infty} \frac{t^k}{k!} R^k$ the usual exponential functions for matrices. 
Additionally, we denote by $\pi_1,\pi_2: \R^{n} \times \R^{n} \to \R^n$ the projections $\pi_1(x,y) = x$ and $\pi_2(x,y) = y$ for $(x,y) \in \R^n \times \R^n$. 
Lastly, we write $\iota: \R^n \to \R^{n} \times \R^n$ for the embedding $\iota(x) = (x,0) \in \R^n \times \R^n$, $x\in \R^n$.

\begin{proposition}\label{prop: coer}
Let the matrices $B,G,C,M,N, F$ be matrices as before and $T>0$. Additionally assume that $C$ is invertible. 
Set 
\begin{equation}\label{eq: defR}
R := \begin{pmatrix} -M^T && - B^T \\ N C^{-1}N^T &&  M \end{pmatrix} \in \R^{2n \times 2n}.
\end{equation}
Assume that the linear mapping 
\begin{equation} \label{eq: invert}
\R^n \ni x \rightarrow \pi_1 (e^{TR}  \iota(x)) + G \pi_2 (e^{TR} \iota(x)) 
\end{equation}
is invertible.

Then there exists $\delta>0$ such that whenever $|x^0| + \| A \|_{L^1} \leq \delta$ then the functional $\tilde{S}_Q$ is weakly coercive on $\mathcal{A}$, i.e., for every sequence $(\gamma^{(k)},\lambda^{(k)}, \mu^{(k)})_k \subseteq \mathcal{A}$ such that it holds $\sup_k \tilde{S}_Q[\gamma^{(k)},\lambda^{(k)}, \mu^{(k)}] < \infty$ there exists a subsequence that is weakly convergent in $\mathcal{A}$.  \\
If $F=0$ the dual functional $\tilde{S}_Q$ is coercive on $\mathcal{A}$ only under the assumption that $C$ is invertible and the invertibility of \eqref{eq: invert}.

\end{proposition}

\begin{remark}
    Let us briefly comment on the invertibility condition \eqref{eq: invert}. We write 
    \[
    \rho := \dot{\gamma} + M^T \gamma + B^T \lambda \text{ and } \sigma := \dot{\lambda} - M \lambda + N\mu = \dot{\lambda} - M \lambda - N C^{-1} N^T\gamma + NC^{-1} ( N^T \gamma + C \mu).
    \]
    Hence, the functions $\gamma, \lambda$ solve the ODE
    \[
    \begin{pmatrix}
        \dot{\gamma} \\ \dot{\lambda}
    \end{pmatrix} = R \begin{pmatrix}
        \gamma \\ \lambda
    \end{pmatrix} + \begin{pmatrix}
        \rho \\ \sigma - NC^{-1}(N^T \gamma + C \mu)
    \end{pmatrix},
    \]
    where $R$ is the matrix in \eqref{eq: defR}. 
    By Proposition \ref{prop: lb g} it seems reasonable that for admissible functions $(\gamma,\lambda,\mu) \in \mathcal{A}$ with a bounded energy $\tilde{S}_Q[\gamma,\lambda,\mu]$ one might hope to control the functions $\rho$, $\sigma$ and $N^T \gamma + C \mu$ so that it holds approximately
    \[
    \begin{pmatrix}
        \gamma(t) \\ \lambda(t)
    \end{pmatrix} \approx e^{tR} \begin{pmatrix}
        \gamma(0) \\ \lambda(0)
    \end{pmatrix}.
    \]
    Then the invertibility of the mapping in \eqref{eq: invert} is exactly the condition that allows to control $\gamma(0)$ through $\lambda(0) = \lambda_0$ using that $\gamma(T) = G \lambda(T)$. In turn this will imply bounds on the functions $\gamma$ and $\lambda$ in $H^1((0,T);\R^n)$.
\end{remark}

\begin{proof}
Let $(\gamma^{(k)},\lambda^{(k)}, \mu^{(k)})_k \subseteq \mathcal{A}$ be a sequence satisfying $\sup_k \tilde{S}_Q[\gamma^{(k)},\lambda^{(k)}, \mu^{(k)}] \leq K$, for some $K>0$.
By the usual weak compactness properties of the spaces $H^1$ and $L^2$ it suffices to prove that the sequence $(\gamma^{(k)},\lambda^{(k)}, \mu^{(k)})_k$ lies in a bounded subset of $H^1(0,T; \R^m) \times H^1(0,T; \R^m) \times L^2((0,T); \R^m)$.

Throughout the proof $\overline{c}>0$ will denote a constant which does not depend on $\gamma^{(k)}$, $\lambda^{(k)}$ nor $\mu^{(k)}$ but may change from line to line.

Defining $\sigma^{(k)} = \dot{\gamma}^{(k)} + M^T\gamma^{(k)} + B^T \lambda^{(k)}$ and $\rho^{(k)} = \dot{\lambda}^{(k)} - M \lambda^{(k)}  + N \mu^{(k)}$ we obtain from Proposition \ref{prop: lb g} for a constant $c>0$ (depending on $a_x, a_p$ and $|F|$)
\begin{align}
&K \geq \tilde{S}_Q[\gamma^{(k)},\lambda^{(k)}, \mu^{(k)}]  \label{eq: est coer dual} \\
\geq &\int_0^T \left( \frac{|\sigma^{(k)}|^2}{c(1 + |\lambda^{(k)}| + |\gamma^{(k)}|)} + \frac{|\rho^{(k)} |^2}{c(1 + |\lambda^{(k)}| + |\gamma^{(k)}|)} + \frac1{2a_u} |N^T\gamma^{(k)} + C \mu^{(k)}|^2 + A\cdot \gamma^{(k)} \right)  \, dt - \gamma^{(k)}(0)\cdot x^0 \nonumber \\
\geq &\int_0^T \left( \frac{|\sigma^{(k)}|^2}{c(1 + |\lambda^{(k)}| +  |\gamma^{(k)}|)} + \frac{|\rho^{(k)} |^2}{c(1 + |\lambda^{(k)}| + |\gamma^{(k)}|)} + \frac1{2a_u} |N^T\gamma^{(k)} + C \mu^{(k)}|^2 \right)  \, dt \nonumber  \\
& - \|A\|_{L^1} \| \gamma^{(k)}\|_{L^{\infty}} - |\gamma^{(k)}(0)| \, |x^0| \nonumber \\
\geq &\int_0^T \left( \frac{|\sigma^{(k)}|^2}{c(1 + |\lambda^{(k)}| +  |\gamma^{(k)}|)} + \frac{|\rho^{(k)} |^2}{c(1 + |\lambda^{(k)}| +  |\gamma^{(k)}|)} + \frac1{2a_u} |N^T\gamma^{(k)} + C \mu^{(k)}|^2 \right)  \, dt - \delta \| \gamma^{(k)} \|_{L^{\infty}}. \nonumber
\end{align}
Next, we write
\begin{equation}\label{eq: rep derivative}
\begin{pmatrix}
    \dot{\gamma}^{(k)} \\ \dot{\lambda}^{(k)}
\end{pmatrix} = \underbrace{\begin{pmatrix} -M^T && - B^T \\ N C^{-1}N^T &&  M \end{pmatrix}}_{= R} \begin{pmatrix} \gamma^{(k)} \\ \lambda^{(k)} \end{pmatrix} - \begin{pmatrix} 0 \\ N C^{-1} ( N^T\gamma^{(k)} + C\mu^{(k)}) \end{pmatrix} +  \begin{pmatrix} \sigma^{(k)} \\ \rho^{(k)} \end{pmatrix}.
\end{equation}
This ODE is solved uniquely by 

\begin{equation} 
\begin{pmatrix} \gamma^{(k)}(t) \\ \lambda^{(k)}(t) \end{pmatrix} = e^{tR} \begin{pmatrix} \gamma^{(k)}(0) \\ \lambda^{(k)}(0) \end{pmatrix} + e^{tR} \int_0^t e^{-sR} \begin{pmatrix} \sigma^{(k)}(s) \\ - N^TC^{-1} ( N \gamma^{(k)}(s) + C\mu^{(k)}(s)) + \rho^{(k)}(s) \end{pmatrix} \, ds \label{eq: sol ode}
\end{equation}
Writing the initial condition in the form $\begin{pmatrix} \gamma^{(k)}(0)& \lambda^{(k)}(0)\end{pmatrix}^T$ = $\begin{pmatrix} 0 & \lambda^{(k)}(0)\end{pmatrix}^T + \begin{pmatrix} \gamma^{(k)}(0) & 0\end{pmatrix}^T$ and using $\gamma^{(k)}(T) = G \lambda^{(k)}(T)$ and $\lambda^{(k)}(0) = \lambda^0$ in \eqref{eq: sol ode} with some rearrangement of terms, we find that
\begin{equation}
\begin{aligned}
&    \pi_1 (e^{TR} \iota (\gamma^{(k)}(0))) - G \pi_2 (e^{TR} \iota(\gamma^{(k)}(0))) \\ 
= & - (\pi_1 - G \circ \pi_2) \left( e^{TR} \begin{pmatrix} 0 \\ \lambda^0 \end{pmatrix} + e^{TR} \int_0^T e^{-sR} \begin{pmatrix} \sigma^{(k)}(s) \\ - N C^{-1} ( N^T\gamma^{(k)}(s) + C\mu^{(k)}(s)) + \rho^{(k)}(s) \end{pmatrix} \, ds \right).  
\end{aligned}
\end{equation}
By the invertibility of the mapping in \eqref{eq: invert} we obtain 
\begin{align}\label{eq: est gamma0}
    |\gamma^{(k)}(0)| 
    \leq \overline{c} \left( |\lambda^0| + \int_0^T |\sigma^{(k)}| + |\rho^{(k)}| + |N^T\gamma^{(k)} + C \mu^{(k)}| \, ds \right)
\end{align}
Combining this with \eqref{eq: sol ode}, we find using \eqref{eq: est coer dual} and H\"older's inequality 
\begin{align*}
    &\| \gamma^{(k)} \|_{L^{\infty}} + \| \lambda^{(k)} \|_{L^{\infty}} \\
    \leq &\overline{c} \left( |\gamma^{(k)}(0)| +  |\lambda^0| + \int_0^T |\sigma^{(k)}| + |\rho^{(k)}| + |N^T\gamma^{(k)} + C \mu^{(k)}| \, ds \right) \\
    \leq &\overline{c} \left(|\lambda^0| + \| N^T \gamma^{(k)} + C \mu^{(k)}\|_{L^2} +  \left( \int_0^T \frac{|\sigma^{(k)}|^2 + |\rho^{(k)}|^2 }{1+|\lambda^{(k)}| + |\gamma^{(k)}|} \, dt \right)^{1/2} (1 + \| \lambda^{(k)} \|_{L^{\infty}} + \| \gamma^{(k)} \|_{L^{\infty}})^{1/2} \right)
    \end{align*}
Using \eqref{eq: est coer dual} and Young's inequality it follows
    \begin{align*}
        &\| \gamma^{(k)} \|_{L^{\infty}} + \| \lambda^{(k)} \|_{L^{\infty}} \\
    \leq \  &\overline{c} \left(|\lambda^0| + (K + \delta \| \gamma^{(k)}\|_{L^{\infty}})^{1/2} +  \left( K + \delta \| \gamma^{(k)}\|_{L^{\infty}} \right)^{1/2}  (1 + \| \lambda^{(k)} \|_{L^{\infty}} + \| \gamma^{(k)} \|_{L^{\infty}})^{1/2}\right) \\
\leq \  &\overline{c} \left( |\lambda^0| + 1 +  K + \delta \| \gamma^{(k)}\|_{L^{\infty}} + \delta^{-1/2}( K + \delta \| \gamma^{(k)} \|_{L^{\infty}} ) + \delta^{1/2} ( 1 + \| \lambda^{(k)} \|_{L^{\infty}} + \| \gamma^{(k)} \|_{L^{\infty}}) \right).
\end{align*}
Hence, if 
\[
\overline{c}(2\delta^{1/2}+\delta) \leq \frac12,
\]
we find 
\[
        \| \gamma^{(k)} \|_{L^{\infty}} + \| \lambda^{(k)} \|_{L^{\infty}} \leq 2 \overline{c} \left( |\lambda^0| + 1 + \delta^{1/2} + K + \delta^{-1/2} K\right).
\]

In particular, $\lambda^{(k)}$ and $\gamma^{(k)}$ are uniformly bounded in $L^{\infty}$. By \eqref{eq: est coer dual} this implies that the functions $\rho^{(k)}$ and $\sigma^{(k)}$ are uniformly bounded in $L^2$. Consequently, 
\begin{align*} 
&\sup_k \| \dot{\gamma}^{(k)} \|_{L^2} + \| \dot{\lambda}^{(k)} \|_{L^2}  \\ \leq &\sup_k \overline{c} \left( \| \lambda^{(k)} \|_{L^2} + \| \gamma^{(k)} \|_{L^2} + \| \rho^{(k)} \|_{L^2} + \| \sigma^{(k)} \|_{L^2} +  \| N^T \gamma^{(k)} + C \mu^{(k)}\|_{L^2} \right) < \infty.
\end{align*}
This shows that the sequences $(\gamma^{(k)})_k, (\lambda^{(k)})_k$ are uniformly bounded in $H^1((0,T);\R^n)$. The uniform boundedness of $(\mu^{(k)})_k$ in $L^2(\Omega;\R^m)$ then follows from \eqref{eq: est coer dual}. 

Let us now consider the special case $F=0$. 
First, note that by Proposition \ref{prop: lb g} it holds in the notation from above similarly to \eqref{eq: est coer dual} that
\begin{equation} \label{eq: est dual F0}
\int_0^T \frac1{2a_u} |\sigma^{(k)}|^2 + \frac1{2a_p} |\rho^{(k)} |^2 + \frac1{2a_u} |N^T\gamma^{(k)} + C \mu^{(k)}|^2  \, dt \leq K + (|x^0| + \| A \|_{L^1}) \| \gamma^{(k)} \|_{L^{\infty}}. 
\end{equation}
By \eqref{eq: est gamma0} it follows that
\[
|\gamma^{(k)}(0)| \leq \overline{c} \left( |\lambda^0| + ( K + (|x^0| + \| A \|_{L^1}) \| \gamma^{(k)} \|_{L^{\infty}})^{1/2} \right).
\]
In turn, using \eqref{eq: sol ode} and \eqref{eq: est dual F0} it follows
\begin{align}
     &\| \gamma^{(k)} \|_{L^{\infty}} + \| \lambda^{(k)} \|_{L^{\infty}} \\
    \leq &\overline{c} \left( |\gamma^{(k)}(0)| +  |\lambda^0| + \int_0^T |\sigma^{(k)}| + |\rho^{(k)}| + |N^T\gamma^{(k)} + C \mu^{(k)}| \, ds \right) \\
    \leq &\overline{c} \left( |\lambda^0| + ( K + (|x^0| + \| A \|_{L^1}) \| \gamma^{(k)} \|_{L^{\infty}})^{1/2} \right).
\end{align}
This implies that $\sup_k \| \gamma^{(k)} \|_{L^{\infty}} + \| \lambda^{(k)} \|_{L^{\infty}} < \infty$. The rest of the proof then concludes analogously to before.

\end{proof}

\begin{remark}
Let us remark that the invertibility assumption for $C$ can be slightly weakened.
Assume that it holds $\text{ker}(C) \subseteq \text{ker}(N)$ and $\text{im}(N^T) \subseteq \text{im}(C)$. Then, it holds for all $x,y \in \R^m$ that $Cx = Cy$ implies that $Nx = Ny$. Hence, for all $z \in \text{im}(C)$ one can define the linear mapping $NC^{-1}$ in a well-defined manner  as $NC^{-1}z = Nx$, where $Cx = z$. By definition it follows that $N C^{-1} C\mu = N\mu$. In addition, since $\text{im}(N^T) \subseteq \text{im}(C)$ also the linear mapping $NC^{-1} N^T$ is well-defined. Using this in the definition of the matrix $R$ in \eqref{eq: defR} the proof above shows that the coercivity of $\tilde{S}_Q$ still holds. 

\end{remark}

In the following we will argue that coercivity of the dual functional $\tilde{S}_Q$ might fail if the mapping \eqref{eq: invert} is not invertible or if $|x^0|$ or $\|A \|_{L^1}$ is too large.

\begin{remark}
\begin{enumerate}
    \item Let $F=0$, $A=0$, $\lambda^0 = 0$ and assume that the mapping \eqref{eq: invert} is not invertible.
    Then there exists $\gamma_0 \in \R^m \setminus \{ 0\}$ such that $x^0 \cdot \gamma_0 \leq 0$ and 
    \[
    \pi_1 (e^{TR} \iota(\gamma_0)) - G \pi_2 (e^{TR} \iota (\gamma_0)) = 0.
    \]
    Then define the sequence of functions
    \[
    \begin{pmatrix} \gamma^{(k)}(t) \\ \lambda^{(k)}(t) \end{pmatrix} = e^{tR} \begin{pmatrix} k \gamma_0 \\ 0  \end{pmatrix}
    \]
    and $\mu^{(k)} = -C^{-1} N^T \gamma^{(k)}$.
    It follows that 
    \[
    \gamma^{(k)}(T) - G\lambda^{(k)}(T)   = \pi_1\left( e^{TR} \begin{pmatrix} k \gamma_0 \\ 0  \end{pmatrix}\right) - G \pi_2\left( e^{TR} \begin{pmatrix} k \gamma_0 \\ 0  \end{pmatrix}\right) = 0.
    \]
    Moreover,
    \[
    \begin{pmatrix} \dot{\gamma}^{(k)}(t) \\ \dot{\lambda}^{(k)}(t) \end{pmatrix} = R \begin{pmatrix} \gamma^{(k)}(t) \\ \lambda^{(k)}(t) \end{pmatrix}
    \]
    and therefore
    \begin{align*}
\dot{\gamma}^{(k)} + M^T \gamma^{(k)} + B^T \lambda^{(k)} = \dot{\lambda}^{(k)} - M \lambda^{(k)} + N^T \mu^{(k)} = 0 \text{ and } N^T \gamma^{(k)} + C\mu^{(k)} = 0,
    \end{align*}
    which implies by Proposition \ref{prop: lb g} that $g(\gamma^{(k)},\dot{\gamma}^{(k)},\lambda^{(k)},\dot{\lambda}^{(k)}, \mu^{(k)}) = 0$, i.e.,
    \[
    \tilde{S}_Q[\gamma^{(k)},\lambda^{(k)},\mu^{(k)}] = \gamma^{(k)}(0)\cdot x^0 = k \gamma_0 \cdot x^0 \leq 0.
    \]
    Hence, $\sup_k \tilde{S}_Q[\gamma^{(k)},\lambda^{(k)},\mu^{(k)}] < \infty$ but the sequence $(\gamma^{(k)},\lambda^{(k)}, \mu^{(k)})_k$ is not bounded in  $H^1(0,T; \R^m) \times H^1(0,T; \R^m) \times L^2((0,T); \R^m)$, i.e., the dual functional $\tilde{S}_Q$ is not coercive.
\item  Let $n=m=2$, $B=C=G= \begin{pmatrix} 1&& 0 \\ 0 && 1 \end{pmatrix}$, $M=0$, $F_{ijr} = \delta_{jr}$, $N = \begin{pmatrix} 1 && 0 \\ 0 && 0\end{pmatrix}$, $\lambda^0 = 0$ and $a_x=a_p=a_u=1$. 
 For $k \in \N$ we set $\lambda^{(k)} = 0$, $\gamma^{(k)}_1 = 0$, $\gamma^{(k)}_2(t) = k( T-t)^2$ and $\mu^{(k)} = 0$.
    Then $\gamma^{(k)}(T) + G \lambda^{(k)}(T) = 0$ and $\lambda^{(k)}(0) = 0 = \lambda^0$. 
    Moreover, $\dot{\gamma}^{(k)}_2 = 2 k (t-T)$ and $N^T\gamma^{(k)} + C \mu^{(k)} = 0$.
    Then compute 
    \begin{align*}
&\tilde{S}_Q[\gamma^{(k)},\lambda^{(k)},\mu^{(k)}] \\ 
= & \int_0^T \sup_{x,p,u\in \R^2} - \dot{\gamma}^{(k)}(t) \cdot x -\frac12(\gamma^{(k)}_2(t) + 1) x^2  - A(t)\cdot \gamma^{(k)}(t) \, dt - \gamma^{(k)}(0)\cdot x^0 \\
= &\int_0^T \frac{\left| \dot{\gamma}^{(k)}  \right|^2}{2(\gamma^{(k)}_2(t) + 1)}  - A(t) \cdot \gamma^{(k)}(t) \, dt - \gamma^{(k)}(0)\cdot x^0 \\
=& \int_0^T \frac{4k^2 |t-T|^2}{2 k (T-t)^2 + 2}  - A_2(t) k (T-t)^2 \, dt - x^0_2 \, k T^2 \\
\leq& kT \left(2 - x^0_2T -  \int_0^T A_2(t) \frac{(t-T)^2}T \, dt\right).
    \end{align*}
Hence, if $x^0_2 \gg 1$ or for specific choices of $A$ with $\|A_2\|_{L^1} \gg 1$ we find that $\tilde{S}_Q[\gamma^{(k)},\lambda^{(k)},\mu^{(k)}] \to -\infty$ for $k\to \infty$. In particular, $\tilde{S}_Q$ is not coercive.

Eventually, we check that the mapping \eqref{eq: invert} is invertible.
Using that $N^T = N$ and $N^2=N$ it holds by definition \eqref{eq: defR} that  $R = \begin{pmatrix}
    0 && - Id_2 \\ N && 0
\end{pmatrix}$, where $Id_2 \in \R^{2\times 2}$ denotes the identity matrix.
Using that $N^2 = N$ one can then show for $l\geq 1$ that
\[
R^{2l} = (-1)^l \begin{pmatrix}
    N && 0 \\ 0 && N
\end{pmatrix} \text{ and } R^{2l+1} = (-1)^l \begin{pmatrix}
    0 && N \\ -N && 0
\end{pmatrix}.
\]
Consequently, it follows 
\begin{align*}
&e^{TR} \\ = &\begin{pmatrix}
    Id_2 && 0  \\ 0 && Id_2
\end{pmatrix} + t \begin{pmatrix}
    0 && - Id_2 \\ N && 0
\end{pmatrix} + \sum_{l=1}^{\infty} (-1)^l \frac{t^{2l}}{(2l)!} \begin{pmatrix}
    N && 0 \\ 0 && N
\end{pmatrix} + \sum_{l=1}^{\infty} (-1)^l \frac{t^{2l+1}}{(2l+1)!} \begin{pmatrix}
    0 && N \\ -N && 0
\end{pmatrix} \\
= &\begin{pmatrix}
    Id_2 && 0  \\ 0 && Id_2
\end{pmatrix} + t \begin{pmatrix}
    0 && - Id_2 \\ N && 0
\end{pmatrix} + (\cos(t)-1) \begin{pmatrix}
    N && 0 \\ 0 && N
\end{pmatrix}+  (\sin(t) - t) \begin{pmatrix}
    0 && N \\ -N && 0
\end{pmatrix}.
\end{align*}
It follows for $x = \begin{pmatrix} x_1 \\ x_2 \end{pmatrix} \in \R^2$ that
\begin{align}
    &\pi_1 (e^{TR}  \iota(x)) + G \pi_2 (e^{TR} \iota(x)) \\
    =& \begin{pmatrix} (\cos(T) - \sin(T) + 2T) x_1 \\ x_2  \end{pmatrix}.
\end{align}
Since $\cos(T) - \sin(T) + 2T \geq 1$ for all $T \geq 0$ it follows that the mapping in \eqref{eq: invert} is invertible.

\end{enumerate}
\end{remark}

On the other hand, if a solution to the primal ODE exists then the dual functional cannot be unbounded from below.

\begin{proposition}
    If a solution to the primal ODE \eqref{eq:primal_QQR}, \eqref{eq:primal_QQR_bc} exists then the dual functional $\tilde{S}_Q$ is bounded from below.
\end{proposition}
\begin{proof}
    Let $x,p,u$ denote the solution to \eqref{eq:primal_QQR} and \eqref{eq:primal_QQR_bc}. Then it holds for $\gamma, \lambda, \mu$ that
    \begin{align*}
        \tilde{S}_Q[\gamma,\lambda,\mu] &= \int_0^T g(\gamma,\dot{\gamma}, \lambda, \dot{\lambda}, \mu) \, - A_i \cdot \gamma_i dt - x^0 \cdot \gamma(0) \\
        &\geq \int_0^T  \Big\{ - \gamma_i  A_i - x_i \dot{\gamma}_i - \gamma_i M_{ij} x_j - \gamma_i N_{i \alpha} u_\alpha - \gamma_i \half F_{irs} x_r x_s \\
        & \qquad \qquad  - p_i \dot{\lambda}_i - \lambda_i B_{ij} x_j + p_j M_{ji} \lambda_i + p_j F_{jri} x_r \lambda_i\\
        & \qquad \qquad +  p_i N_{i \alpha} \mu_\alpha -  u_\beta C_{\beta \alpha} \mu_\alpha\\
        & \qquad \qquad - \sup{x,p,u} \half \Big( a_x |x |^2 + a_u |u |^2 + a_p |p |^2 \Big) \ \Big\} \, dt - \gamma_i|_0 x^0_i. \\
        &= - \int_0^T \half \Big( a_x  |x |^2 + a_u |u |^2 + a_p |p |^2 \Big) \, dt.
    \end{align*}
\end{proof}
\section{The forced Linear-Quadratic Regulator (LQR)}\label{sec:LQR}
The dual formulation of the LQR problem is obtained from the dual QQR by setting $F = 0$. Define
\begin{equation*}
    \mathcal{J}^l \Big|_\dee : = \begin{Bmatrix} \calX|_\dee  \\ \\ \calP|_\dee  \end{Bmatrix} \; \qquad \K^l \Big|_\dee := \begin{bmatrix} a_x I & 0  \\
                                                                  &   \\
                                                                  0 & a_p I \end{bmatrix}; \qquad I \mbox{ is the } n \times n \mbox{ identity matrix}.\notag\\
\end{equation*}
The \emph{LQR dual functional} is then given by
\begin{subequations}
    \label{eq:LQR_dual}
    \begin{align}
        S_Q^l[D] & := \half \int_0^T \bigg( \mathcal{J}^l \Big|_{(\dee,\bar{U})} \cdot \K^l \Big|_\dee^{-1} \mathcal{J}^l \big|_{(\dee,\bar{U})} \ + \ \frac{1}{a_u} \calU \big|_\dee \cdot \calU \big|_\dee \bigg) \, dt \notag \\
        & \quad - \int_0^T \Big\{ \bar{x} \cdot \calX |_\dee  \ + \ \bar{p} \cdot \calP |_\dee  \ + \ \bar{u} \cdot \calU |_\dee \ + \ A \cdot \gamma \Big\} \, dt \ - \  \gamma|_0 \cdot x^0 \notag \\
        & \qquad \qquad = \tag{\ref{eq:LQR_dual}}\\
        & \quad \half \int_0^T \bigg( \frac{1}{a_x} \calX |_\dee \cdot \calX |_\dee \ + \ \frac{1}{a_p} \calP |_\dee \cdot \calP |_\dee \ + \  \frac{1}{a_u} \calU |_\dee \cdot \calU |_\dee \bigg) \, dt  \notag\\
        & \quad - \int_0^T \Big\{ \bar{x} \cdot \calX |_\dee  \ + \ \bar{p} \cdot \calP |_\dee  \ + \ \bar{u} \cdot \calU |_\dee \ + \ A \cdot \gamma \Big\} \, dt \ - \  \gamma|_0 \cdot x^0. \notag
    \end{align}
\end{subequations}
Its first variation, in a direction $\delta D$,  generates the weak form of the LQR problem \eqref{eq:primal_QQR} with $F = 0$ and is given by
\begin{equation}
    \label{eq:LQR_dual_first_var}
    \begin{aligned}
        \delta S^l_Q\big|_{\delta D} [D] & = \int_0^T  \left( \frac{1}{a_x}  \calX|_{\delta \dee} \cdot \calX|_\dee \ + \ \frac{1}{a_u} \calU|_{\delta \dee} \cdot \calU|_\dee
    + \ \frac{1}{a_p} \calP|_{\delta \dee} \cdot \calP|_\dee \right) \, dt \\
    &  \quad - \int_0^T \Big( \bar{x} \cdot \calX|_{\delta \dee} \ + \ \bar{u} \cdot \calU|_{\delta \dee} \ + \bar{p} \cdot \calP|_{\delta \dee} + A \cdot \delta \gamma \Big) \, dt \  - \ \delta \gamma|_0 \cdot x^0;
    \end{aligned}
\end{equation}
 Noting that $\dee$ is naturally viewed as a linear operator on the space, say 
 \[
 Y = \big\{ D \, | \, D:[0,T] \to \R^{n+m+m}, (\gamma, \lambda) \in H^1 \big([0,T],\R^{2n}\big), \mu \in L^{2}([0,T],\R^m) , D \mbox{ satisfies } \eqref{eq:gamma_bc} \big\}
 \]
 of admissible dual functions $D$,  and $(\calX, \calU, \calP)$ is a linear function on $\dee$, the second variation of $S^l_Q$, in the pair of directions $(\delta D, d D) \in Y \times Y$, is a \emph{symmetric, positive semi-definite bilinear} operator in $(\delta D, d D)$ independent of $D$,  and given by
 \begin{equation}\label{eq:def_L}
     d \delta S^l_Q\big|_{(\delta D, dD)} [D] \ = \ \int_0^T  \left( \frac{1}{a_x}  \calX|_{\delta \dee} \cdot \calX|_{d \dee} \ + \ \frac{1}{a_u} \calU|_{\delta \dee} \cdot \calU|_{d \dee}
    + \ \frac{1}{a_p} \calP|_{\delta \dee} \cdot \calP|_{d \dee} \right) \, dt. =: \mathbb{L} (\delta D, d D).
 \end{equation}
 
At a critical point $D$ of the functional the first variation must vanish for all variations $\delta D$ consistent with the boundary conditions \eqref{eq:gamma_bc}. If $(\bar{x},\Bar{u},\bar{p})$ is a solution to (\ref{eq:primal_QQR}-\ref{eq:primal_QQR_bc}) for $F = 0$, then the second line in \eqref{eq:LQR_dual_first_var} must vanish for all admissible variations. Then it is clear that $D = (\gamma, \mu, \lambda) = 0$ is a critical point. In general, the choice of a guiding base state $\bar{U} = (\bar{x}, \bar{u}, \bar{p})$ act as forcing to the problem, as do the boundary conditions \eqref{eq:gamma_bc}.

Let us define a linear operator on $Y$
\[
\mathsf{l}(D;A,x^0,\bar{U}) := \int_0^T \Big\{ \bar{x} \cdot \calX |_\dee  \ + \ \bar{p} \cdot \calP |_\dee  \ + \ \bar{u} \cdot \calU |_\dee \ + \ A \cdot \gamma \Big\} \, dt \ + \  \gamma|_0 \cdot x^0.
\]
Invoking a finite set of linearly independent functions $\mathsf{B} = \{ \phi_A \in Y, A = 1, \ldots, N \}$ and writing 
\[
D = D_B \phi_B, \qquad \delta D = \delta D_A \phi_A,
\]
a discrete linear algebra based approximation of the  LQR primal problem \eqref{eq:primal_QQR} with $F = 0$ is given by the following finite dimensional approximation of \eqref{eq:LQR_dual_first_var} (using \eqref{eq:def_L}):
\[
M_{AB} D_B = f_A, \qquad \mathbb{L}(\phi_A, \phi_B) =: M_{AB}; \qquad f_A =: \mathsf{l}(\phi_A; A,x^0, \bar{U}); \quad A, B = 1, \ldots, N.
\]

\emph{We note that the matrix $M$ is symmetric and does not depend on the initial condition $x^0$}. Thus, given the problem \eqref{eq:primal_QQR} with $F= 0$, the time interval $[0,T]$, and the basis $\mathsf{B}$, an \textit{SVD} decomposition (or \textit{LU}, when nonsingular, or any other preferred linear algebraic decomposition)  of the matrix $M$ can be precomputed off-line and stored. This can be used in conjunction with the vector $f$, which depends on the initial condition $x^0$, to generate the state, co-state, and control approximations for the LQR problem for each specific $x^0$. In the presence of nonuniqueness of solutions to the primal control problem, the base state $\bar{U}$ employed enforces the obtained solutions/approximations to be closest to it.

Of course, because the LQR dual functional \eqref{eq:LQR_dual} is convex, the powerful methods of convex optimization \cite{nesterov2018lectures} can be brought to bear on the discrete problem.

The classical method for solving LQR systems relies on a one-time solution of a nonlinear Riccati equation to generate the feedback control (with the matrix $C$ assumed invertible): one considers the LQR optimality conditions \eqref{eq:primal_QQR} with $F= 0$ and $A = 0$,
\begin{subequations}
    \allowdisplaybreaks
    \begin{align}
        \dot{x} & = Mx + Nu \label{eq:LQR1}\\
        \dot{p} & = Bx - M^Tp \label{eq:LQR2}\\
        0 & = N^Tp - Cu \label{eq:LQR3}
    \end{align}
\end{subequations}
with the boundary conditions \eqref{eq:primal_QQR_bc}, along with the ansatz
\begin{equation}\label{eq:riccati_ansatz}
p|_t := K|_t \, x|_t \qquad t \mapsto K|_t \in \R^{n \times n}_{sym}
\end{equation}
and assuming (\ref{eq:LQR1}-\ref{eq:LQR3}-\ref{eq:riccati_ansatz}) are satisfied, one chooses the function $K$ such that \eqref{eq:LQR2} is also satisfied (for alternate motivation based on the value function of the HJB procedure, see, e.g., \cite[Sec.~5.2.3]{evans1983introduction}). Thus,
\[
u = C^{-1}N^TKx \qquad \mbox{and} \qquad \dot{x} = Mx + NC^{-1}N^TKx,
\]
so that \eqref{eq:LQR2} corresponds to the statement
\[
\dot{K}x + K(Mx + NC^{-1}N^TKx) - Bx + M^TKx  = 0,
\]
which is satisfied if $K$ satisfies the (nonlinear) Riccati equation
\[
\dot{K}|_t + (K|_t M + M^T K|_t) + K|_t N C^{-1}N^TK|_t - B = 0; \qquad K|_T = - G,
\]
with the feedback control given by
\begin{equation}\label{eq:LQR_feedback}
u|_t = C^{-1}N^TK|_t \, x|_t.
\end{equation}
Then, for each specific problem corresponding to initial condition $x(0) = x^0$, one substitutes \eqref{eq:LQR_feedback} into \eqref{eq:LQR1} to solve for the state response $t \mapsto x(t)$.

In our scheme, the one-off solution procedure involves a linear problem. The analog of the solution of the Riccati problem corresponds to obtaining the solution operator for the \emph{linear}, second-order bvp for $(\gamma, \mu, \lambda)$ subject to the boundary conditions \eqref{eq:lambda_bc}-\eqref{eq:linear_dual_bc}-\eqref{eq:primal_QQR_bc} as a function of the boundary data $(x^0, \lambda^0)$ (an outline of the computational protocol for the corresponding discrete case has been discussed above). Admittedly, it is not in feedback control form, but it does not require 
\begin{itemize}
    \item the matrix $C$ to be invertible, and 
    \item the forcing vector $A$ to vanish. 
\end{itemize}

The formal non-requirement of the invertibility of the control cost matrix $C$ in the dual formulation of the LQR problem is intriguing and it remains to be seen whether a rigorous existence theorem can be obtained in the absence of this invertibility (we note that even the formal derivation of the matrix Riccati equation for the LQR problem requires this invertibility).
\subsection{An Example}
In this section we derive an explicit solution for the simplest LQR problem by the proposed dual scheme. Consider the problem
\begin{equation*}
    \begin{aligned}
        \mbox{maximize } P[u] & = - \int_0^T x^2(t) + u^2(t) \, dt \\
        \mbox{subject to } \dot{x}(t) & = x(t) + u(t) \\
         x(0) & = - x^0,     
    \end{aligned}
\end{equation*}
and $x^0$ is specified. Here, $n,m = 1$.
Thus, within our setup
\[
f(x,u) = x + u; \qquad G = 0; \qquad r(x,u) = x^2 + u^2,
\]
and the primal system of equations to be solved is
\begin{subequations}\label{eq:primal_ex}
    \begin{align}
        \dot{x} - x - u & = 0 \notag\\
        \dot{p} - 2x + p & = 0 \notag\\
        - p + 2 u & = 0 \tag{\ref{eq:primal_ex}}\\
        p(T) & = 0 \notag\\
        x(0) & = - x^0. \notag
    \end{align}
\end{subequations}
Choosing
\[
Q(x,u,p) = \half (x^2 + u^2 + p^2),
\]
the pre-dual functional is given by
\begin{subequations}
    \begin{align}
        \widehat{S}_Q & = \int_0^T \bigg( - x \dot{\gamma} - \gamma x - \gamma u 
                    - p \dot{\lambda} - 2 \lambda x + \lambda p - p \mu + 2 \mu u \notag\\
                      & \qquad \qquad - \half x^2 - \half u^2 - \half p^2 \bigg) \, dt \quad - \quad \gamma(0) x^0 \notag
    \end{align}
\end{subequations}
with Lagrangian
\[
\scl_Q = x \bigg(- \dot{\gamma} - \gamma - 2 \lambda - \half x \bigg) + u \bigg( - \gamma + 2 \mu - \half u \bigg) + p \bigg( - \dot{\lambda} + \lambda - \mu - \half p \bigg)
\]
and the DtP mapping
\begin{subequations}
    \begin{align}
        \p_x \scl_Q & = 0 : \qquad x^Q = - (\dot{\gamma} + \gamma + 2 \lambda) \notag\\
        \p_u \scl_Q & = 0:  \qquad u^Q = -(\gamma - 2 \mu) \notag\\
        \p_p \scl_Q & = 0: \qquad p^Q = -(\dot{\lambda} - \lambda + \mu). \notag
    \end{align}
\end{subequations}
The corresponding set of governing equations and boundary conditions for the dual problem is given by
\begin{subequations}
    \begin{align}
        \dot{x^Q} - x^Q - u^Q  = -(\ddot{\gamma} + 2 \dot{\lambda}  - 2 \gamma - 2 \lambda + 2 \mu) & = 0 \notag\\
        \dot{p^Q} - 2 x^Q + p^Q  = -(\ddot{\lambda} + \dot{\mu} - 2 \dot{\gamma} - 2 \gamma - 5 \lambda + \mu) & = 0 \notag\\
        - p^Q + 2 u^Q  = -(- \dot{\gamma} + 2 \gamma + \lambda - 5 \mu) & = 0 \label{eq:ex_mu}\\
        \notag \\
        x^Q(0) + x^0  = -(\dot{\gamma}(0) + \gamma(0) + 2 \lambda(0)) + x^0 & = 0 \notag\\
        p^Q(T)  = -(\dot{\lambda}(T) - \lambda(T) + \mu(T)) & = 0 \notag\\
         \qquad \gamma(T) & = 0 \notag\\
         \qquad \lambda(0) & = \lambda^0. \notag
    \end{align}
\end{subequations}
Eliminating $\mu$ using \eqref{eq:ex_mu},
\[
\mu = \frac{1}{5} \big( - \dot{\lambda} + 2 \gamma + \lambda \big),
\]
one obtains the following constant-coefficient, linear, second order reduced system
\begin{equation*}
    \begin{aligned}
        \begin{bmatrix}
            1 & & 0 \\
            \\
            0 & & \frac{4}{5}
        \end{bmatrix}
        \begin{Bmatrix}
            \ddot{\gamma} \\
            \\
            \ddot{\lambda}
        \end{Bmatrix}
        + \begin{bmatrix}
            0 & & \frac{8}{5} \\
                 \\
           - \frac{8}{5} & & 0
          \end{bmatrix}
           \begin{Bmatrix}
            \dot{\gamma}\\
                 \\
            \dot{\lambda}
            \end{Bmatrix}   
        + \begin{bmatrix}
            - \frac{6}{5} & & - \frac{8}{5} \\
            \\
            - \frac{8}{5} & & - \frac{24}{5}
          \end{bmatrix}
          \begin{Bmatrix}
            \gamma\\
                 \\
            \lambda
            \end{Bmatrix}
            =
            \begin{Bmatrix}
                0\\
                 \\
                0
            \end{Bmatrix}
    \end{aligned}
\end{equation*}
with boundary conditions
\begin{equation*}
\begin{aligned}
   \dot{\gamma}(0) + \gamma(0) + 2 \lambda(0) & = x^0 \\
   \frac{4}{5} \dot{\lambda}(T) - \frac{4}{5} \lambda(T) + \frac{2}{5}\mu(T) & = 0 \\
         \qquad \gamma(T) & = 0 \\
         \qquad \lambda(0) & = \lambda^0. 
\end{aligned}
\end{equation*}
The general solution involves two characteristic times (roots) $\pm \sqrt{2}$ each of multiplicity two (so the dual solution is not a linear combination of pure exponentials in time). We use the symbolic mathematics software {\sf Mathematica} \cite{Mathematica} to obtain the explicit solution listed in the Appendix. The explicit forms are not instructive except for the following fact:
\begin{itemize}
    \item the solution for $\left( x^Q, u^Q \right)$ does not depend on (the arbitrary choice) of $\lambda^0$ whereas the solution for the dual functions $(\gamma, \mu, \lambda)$ does; the former condition is necessary when the primal problem has unique solutions. 

    Indeed, we show below that the primal problem in this case has a unique solution.
\end{itemize}
The system \eqref{eq:primal_ex} can also be written as
\[
\begin{pmatrix} \dot{x}\\\dot{p} \end{pmatrix} = \underbrace{\begin{pmatrix} 1 && \frac12 \\ 2 && -1 \end{pmatrix}}_{=:A} \begin{pmatrix} x \\ p \end{pmatrix}, x(0) = -x^0, p(T) = 0.
\]
A solution (if it exists) to the ODE above is of the form 
\[
\begin{pmatrix}
    x \\ p
\end{pmatrix} = e^{At} \begin{pmatrix}
    \bar{x} \\ \bar{p}
\end{pmatrix},
\]
where $\bar{x}, \bar{p} \in\R$ need to be determined. It follows directly that $\bar{x} = -x^0$. The equation that determines $\bar{p}$ is
\[
0 = e_2 \cdot \left(e^{AT} \begin{pmatrix}
    \bar{x} \\ \bar{p}
\end{pmatrix} \right) = (e^{AT})_{21} \bar{x} + (e^{AT})_{22} \bar{p}.
\]
Hence, if $(e^{AT})_{22} \neq 0$ then the value of $\bar{p}$ is uniquely determined. If $(e^{AT})_{22}=0$ then $(e^{AT})_{21} \neq 0$ (since $e^{AT}$ is always invertible) and there are either infinitely many solutions for $\bar{p}$  (if $\bar{x} = 0$) or no solution for $\bar{p}$ (if $\bar{x}\neq 0$). 

In this particular case it can be checked that $(e^{AT})_{22} \neq 0$ and the solution to \eqref{eq:primal_ex} is unique.

\begin{appendix} 
\section{Appendix: Explicit primal and dual solutions for \eqref{eq:primal_ex}}

\begin{subequations}
\allowdisplaybreaks
    \begin{align}
     \bullet \quad x^Q(t)  &= \frac{x_{num}(t)}{x_{den}(t)} \qquad \mbox{where}\\
     \notag \\
x_{num}(t) & = e^{-\sqrt{2} t} \  \big( -x^0 \big)  \ \bigg( \ -\left(\sqrt{2}-2\right) e^{2 \sqrt{2} (t+2 T)} +2 \left(\sqrt{2}+2\right) e^{2 \sqrt{2} (t+T)} \notag\\
& \qquad \qquad \qquad \ +\left(7 \sqrt{2}+10\right) e^{2 \sqrt{2} t} -2 \left(\sqrt{2}-2\right) e^{4 \sqrt{2} T} \notag\\
 &  \qquad \qquad \qquad \qquad  \qquad +\left(\sqrt{2}+2\right) e^{2 \sqrt{2} T} +\left(10-7 \sqrt{2}\right) e^{6 \sqrt{2} T} \ \bigg) \notag\\
  x_{den} (t)  & = \bigg(\left(\sqrt{2}-2\right) e^{2 \sqrt{2} T}-\sqrt{2}-2\bigg) \bigg(-2 e^{2 \sqrt{2} T}+\left(2 \sqrt{2}-3\right) e^{4 \sqrt{2} T}-2 \sqrt{2}-3 \bigg); \notag
    \end{align}
\end{subequations}
\begin{equation*}
    \begin{aligned}
    \\
      \bullet \quad   u^Q(t) = \frac{e^{-\sqrt{2} t} \ \big(- x^0 \big) \ \left(2 \sqrt{2} e^{2 \sqrt{2} T}+\left(3 \sqrt{2}-4\right) e^{4 \sqrt{2} T}+3 \sqrt{2}+4\right) \left(e^{2 \sqrt{2} t}-e^{2 \sqrt{2} T}\right)}{\left(\left(\sqrt{2}-2\right) e^{2 \sqrt{2} T}-\sqrt{2}-2\right) \left(-2 e^{2 \sqrt{2} T}+\left(2 \sqrt{2}-3\right) e^{4 \sqrt{2} T}-2 \sqrt{2}-3\right)};
    \end{aligned}
\end{equation*}
\begin{equation*}
    \begin{aligned}
    \\
       \bullet \quad  p^Q(t) = \frac{2 e^{-\sqrt{2} t} \ \big(- x^0 \big) \ \left(2 \sqrt{2} e^{2 \sqrt{2} T}+\left(3 \sqrt{2}-4\right) e^{4 \sqrt{2} T}+3 \sqrt{2}+4\right)  \left(e^{2 \sqrt{2} t}-e^{2 \sqrt{2} T}\right)}{\left(\left(\sqrt{2}-2\right) e^{2 \sqrt{2} T}-\sqrt{2}-2\right) \left(-2 e^{2 \sqrt{2} T}+\left(2 \sqrt{2}-3\right) e^{4 \sqrt{2} T}-2 \sqrt{2}-3\right)}.
    \end{aligned}
\end{equation*}

It has been verified (in {\sf Mathematica}) that \eqref{eq:primal_ex} is satisfied by the formulae $x = x^Q$,  $u = u^Q$, $p = p^Q$, where the righ-hand-sides are given above. 

The \emph{$\lambda^0$-dependent} solutions to the dual problem are given by 
\begin{subequations}
\allowdisplaybreaks
    \begin{align}
    \bullet \quad \gamma(t) & = \frac{\gamma_{num}(t)}{\gamma_{den}(t)} \qquad \mbox{where} \notag\\
    \notag\\
    \gamma_{num}(t) & = e^{-\sqrt{2} t} \bigg( \ e^{6 \sqrt{2} T} \left(\left(6 \sqrt{2}-8\right) \lambda^0+\left(\left(4 \sqrt{2}-6\right) t-5 \sqrt{2}+6\right) x^0\right) \notag\\
    & \qquad \qquad \quad +e^{2 \sqrt{2} (t+T)} \left(-4 \sqrt{2} \lambda^0-4 t x^0+4 T x^0+6 \sqrt{2} x^0\right) \notag\\
    & \qquad \qquad \quad   + e^{2 \sqrt{2} (t+T)} \left(-4 \sqrt{2} \lambda^0-4 t x^0+4 T x^0+6 \sqrt{2} x^0\right) \notag\\
    & \qquad \qquad \quad +e^{2 \sqrt{2} (t+2 T)} \left(\left(8-6 \sqrt{2}\right) \lambda^0+x^0 \left(\left(4 \sqrt{2}-6\right) t-8 \sqrt{2} T+12 T+5 \sqrt{2}-6\right)\right) \notag\\
    & \qquad \qquad \quad  +e^{2 \sqrt{2} T} \left(\left(6 \sqrt{2}+8\right) \lambda^0-x^0 \left(\left(4 \sqrt{2}+6\right) t-4 \left(2 \sqrt{2}+3\right) T+5 \sqrt{2}+6\right)\right) \quad \bigg) \notag\\
    \gamma_{den}(t) & = \left(\left(\sqrt{2}-2\right) e^{2 \sqrt{2} T}-\sqrt{2}-2\right) \left(-2 e^{2 \sqrt{2} T}+\left(2 \sqrt{2}-3\right) e^{4 \sqrt{2} T}-2 \sqrt{2}-3\right) \notag
    \end{align}
\end{subequations}
\begin{subequations}
\allowdisplaybreaks
    \begin{align}
     \bullet \quad \lambda(t) & = \frac{\lambda_{num}(t)}{\lambda_{den}(t)} \qquad \mbox{where} \notag\\
    \notag\\
      \lambda_{num} (t) & =  e^{-\sqrt{2} t}\bigg( \ e^{6 \sqrt{2} T} \left(\left(10-7 \sqrt{2}\right) \lambda^0+\left(7-5 \sqrt{2}\right) t x^0\right) \notag\\
           & \qquad \qquad \quad +e^{2 \sqrt{2} t} \left(\left(7 \sqrt{2}+10\right) \lambda^0+\left(5 \sqrt{2}+7\right) t x^0\right) \notag\\
           & \qquad \qquad \quad -2 e^{4 \sqrt{2} T} \left(\left(\sqrt{2}-2\right) \lambda^0+\left(\sqrt{2}-1\right) x^0 (t-T-1)\right) \notag\\
           & \qquad \qquad \quad +2 e^{2 \sqrt{2} (t+T)} \left(\left(\sqrt{2}+2\right) \lambda^0+\left(\sqrt{2}+1\right) x^0 (t-T-1)\right) \notag\\
           & \qquad \qquad \quad  +e^{2 \sqrt{2} (t+2 T)} \left(\left(\sqrt{2}-1\right) x^0 (t-2 (T+1))-\left(\sqrt{2}-2\right) \lambda^0\right) \notag\\
           & \qquad \qquad \quad +e^{2 \sqrt{2} T} \left(\left(\sqrt{2}+2\right) \lambda^0-\left(\sqrt{2}+1\right) x^0 (t-2 (T+1))\right) \ \bigg) \notag
    \end{align}
\end{subequations}
\begin{equation*}
    \lambda_{den}(t) = \left(\left(\sqrt{2}-2\right) e^{2 \sqrt{2} T}-\sqrt{2}-2\right) \left(-2 e^{2 \sqrt{2} T}+\left(2 \sqrt{2}-3\right) e^{4 \sqrt{2} T}-2 \sqrt{2}-3\right)
\end{equation*}
\begin{subequations}
\allowdisplaybreaks
    \begin{align}
    \bullet \quad \mu(t) & = \frac{\mu_{num}(t)}{\mu_{den}(t)} \qquad \mbox{where} \notag\\
    \notag \\
    \mu_{num}(t) & = e^{-\sqrt{2} t}\bigg( \ e^{6 \sqrt{2} T} \left(\left(3 \sqrt{2}-4\right) \lambda^0+\left(2 \sqrt{2}-3\right) t x^0-\sqrt{2} x^0+x^0\right) \notag\\
        & \qquad \qquad \quad +e^{2 \sqrt{2} t} \left(\left(-\left(2 \sqrt{2}+3\right) t+\sqrt{2}+1\right) x^0-\left(3 \sqrt{2}+4\right) \lambda^0\right) \notag\\
        & \qquad \qquad \quad +2 e^{4 \sqrt{2} T} \left(\sqrt{2} \lambda^0-x^0 \left(t-T+\sqrt{2}\right)\right) \notag\\
        & \qquad \qquad \quad -2 e^{2 \sqrt{2} (t+T)} \left(\sqrt{2} \lambda^0-x^0 \left(-t+T+\sqrt{2}\right)\right) \notag\\
        & \qquad \qquad \quad +e^{2 \sqrt{2} T} \left(\left(3 \sqrt{2}+4\right) \lambda^0-x^0 \left(\left(2 \sqrt{2}+3\right) t-4 \sqrt{2} T-6 T+\sqrt{2}+1\right)\right) \notag\\
        & \qquad \qquad \quad +e^{2 \sqrt{2} (t+2 T)} \left(\left(4-3 \sqrt{2}\right) \lambda^0+x^0 \left(\left(2 \sqrt{2}-3\right) t-4 \sqrt{2} T+6 T+\sqrt{2}-1\right)\right) \ \bigg) \notag\\
        \mu_{den}(t) & = \left(\left(\sqrt{2}-2\right) e^{2 \sqrt{2} T}-\sqrt{2}-2\right) \left(-2 e^{2 \sqrt{2} T}+\left(2 \sqrt{2}-3\right) e^{4 \sqrt{2} T}-2 \sqrt{2}-3\right). \notag
    \end{align}
\end{subequations}
\end{appendix}
\section*{Acknowledgments}
AA thanks Bob Kohn and Vladimir Sverak for separately raising the question of the connection of the Pontryagin Maximum Principle to the dual variational methodology that forms the basis of this work, which led to the exploration of the topic presented here by us. It is also a pleasure to acknowledge discussions with Soummya Kar on aspects of Control Theory and with Mario Berges. A Simons Pivot Fellowship grant \# 983171 to AA is acknowledged.

\printbibliography
\end{document}

%% file: temp-definitions.tex
\usepackage{amsmath}
\usepackage{amsbsy}
\usepackage{amssymb}
\usepackage{amscd}
\usepackage{amsfonts}
\usepackage{isomath}

\newcommand{\R}{\mathbb R}
\newcommand{\N}{\mathbb N}

\newcommand{\K}{\mathbb{K}}

\newcommand{\beq}{\begin{equation}}
\newcommand{\eeq}{\end{equation}}
\newcommand{\beqs}{\begin{eqnarray}}
\newcommand{\eeqs}{\end{eqnarray}}
\newcommand{\beql}{\begin{equation} \label}
\newcommand{\half}{\frac{1}{2}}

\newcommand{\calP}{{\cal P}}

\newcommand{\calU}{{\cal U}}

\newcommand{\calX}{{\cal X}}

\newtheorem{proposition}{Proposition}[section]
\newtheorem{remark}{Remark}[section]

\newcommand{\p}{\partial}
\newcommand{\dee}{\mathcal{D}}
\newcommand{\scl}{\mathcal{L}}
\newcommand{\sch}{\mathcal{H}}